\documentclass[moor,nonblindrev]{informs1} 

\usepackage{natbib}
 \NatBibNumeric
 \def\bibfont{\small}%
 \bibpunct[, ]{[}{]}{,}{n}{}{,}%

\usepackage[colorlinks=true,breaklinks=true,bookmarks=true,urlcolor=blue,
     citecolor=blue,linkcolor=blue,bookmarksopen=false,draft=false]{hyperref}
     
\usepackage{amsmath,epsfig,graphicx,color, float,amssymb}
\usepackage{subcaption}
\usepackage{relsize}
\usepackage{comment}
\usepackage{pdfpages}
\usepackage{graphicx}
\usepackage{mwe}
\usepackage{algorithm}
\usepackage{algorithmicx}
\usepackage{algpseudocode}
\usepackage{natbib}
\usepackage{dsfont}
\usepackage{bm}

\definecolor{darkblue}{rgb}{.2, 0.2,.8}
\definecolor{darkgreen}{rgb}{0,0.5,0.3}
\definecolor{darkred}{rgb}{.8, .1,.1}
\newcommand{\red}{\color{black}}

\newcommand{\ii}{\mbox{i}}
\newcommand{\dd}{\mathrm{d}}

\newcommand{\bfS}{\mat{S}}
\newcommand{\bfs}{\vect{s}}
\newcommand{\bfe}{\vect{e}}

\newcommand{\E}{\mathbb{E}}
\renewcommand{\P }{{\mathbb P}}

\newcommand{\vect}[1]{\pmb{#1}}
\newcommand{\mat}[1]{\boldsymbol{\bm #1}}
\newcommand{\JJ}[1]{\mathcal{J}^{(#1)}}

\newcommand{\Pg}{\P_{\mathcal{G}_n}}

\def\EMAIL#1{\href{mailto:#1}{#1}}

\TheoremsNumberedThrough     

\EquationsNumberedThrough    

\begin{document}


\RUNAUTHOR{Bladt and Peralta}

\RUNTITLE{Strong approximations to time-inhomogeneous MJPs}

\TITLE{Strongly convergent homogeneous approximations to inhomogeneous Markov jump processes and applications}

\ARTICLEAUTHORS{%
\AUTHOR{Martin Bladt}
\AFF{Faculty of Mathematical Sciences\\
University of Copenhagen\\
 \EMAIL{martinbladt@math.ku.dk}}
\AUTHOR{Oscar Peralta}
\AFF{School of Operations Research and Information Engineering\\Cornell University\\ \EMAIL{op65@cornell.edu}}
} 

\ABSTRACT{%
The study of time-inhomogeneous Markov jump processes is a traditional topic within probability theory that has recently attracted substantial attention in various applications. However, their flexibility also incurs a substantial mathematical burden which is usually circumvented by using well-known generic distributional approximations or simulations. This article provides a novel approximation method that tailors the dynamics of a time-homogeneous Markov jump process to meet those of its time-inhomogeneous counterpart on an increasingly fine Poisson grid. {   Strong convergence of the processes in terms of the Skorokhod $J_1$ metric is established, and convergence rates are provided. Under traditional regularity assumptions, distributional convergence is established for unconditional proxies, to the same limit.} Special attention is devoted to the case where the target process has one absorbing state and the remaining ones transient, for which the absorption times also converge. Some applications are outlined, such as  univariate hazard-rate density estimation, ruin probabilities, and multivariate phase-type density evaluation.
}%



\maketitle

%


By Markov jump process (MJP), here we mean a stochastic process in $\mathbb{R}_+:=[0,\infty)$ evolving on a discrete state-space for which the future and past are independent, given the present state. When the value of the present time is immaterial, we speak of a \textit{time-homogeneous} MJP, and otherwise, it is customary to refer to it as a \textit{time-inhomogeneous} MJP. While the theory of time-inhomogeneous MJPs may be considered a standard topic in the field of probability, a large portion of the applied probability literature focuses in the time-homogeneous subclass only (see e.g. \cite{asmussen2003applied} and references therein). This is due, in part, to the simplicity and tractability of the latter. For instance, a time-homogeneous MJP has exponentially distributed sojourn times, while the time-inhomogeneous MJP has sojourn times which generally cannot be described as easily. Moreover, the transitions probabilities of time-homogeneous MJPs are characterized by matrix-exponential functions, whereas the ones associated with time-inhomogeneous MJPs are in general defined through product integrals, the latter being considerably more sophisticated than its matrix-exponential counterpart (see e.g. \cite{Dollard1984ProductIW}). Nevertheless, in recent years, time-inhomogeneous MJPs have gathered considerable attention in the stochastic modelling community due to their ability to describe discrete systems whose switching rates vary with time (\cite{telek2004analysis,cloth2007computing,frenkel2009non,bladt2020matrix}). We believe that, to further enhance the use of time-inhomogeneous MJPs in applied contexts, it is vital to establish robust means to analyze the paths and descriptors of time-inhomogeneous MJPs employing simple techniques akin to those available for the time-homogeneous case.

On the other hand, the use of strongly convergent approximations in probability has proven to be a fruitful technique which elegantly yields powerful results. For example, consider Donsker's invariance principle, which states that (under certain conditions) the \emph{law} of a properly scaled random walk is guaranteed to converge to the \emph{law} of a standard Brownian motion. Instead of focusing on a purely weakly convergent result, Strassen built, in \cite{strassen1964invariance}, a sequence of strongly convergent scaled random walks which yield improvements of the law of iterated logarithm that were otherwise not available. Important modern probability topics such as rough path analysis \cite{friz2020course} are rooted in the pioneering work of Wong and Zakai \cite{wong1965relation}, based on studying solutions of stochastic differential equations via their strongly convergent approximations. In applied settings, pathwise constructions have also been successfully used, for instance, to study Markov jump processes that converge to solutions of an ODE \cite{kurtz1971limit}, to establish the convergence of empirical and quantile processes in statistics \cite{csorgo2014strong}, or to analyze the stability of multiclass queueing networks via their strong fluid or diffusive limits \cite{mandelbaum1995strong,mandelbaum1998strong,choudhury1997fluid}. Particularly, in this latter topic, the monograph \cite{whitt2002stochastic} provides a complete and excellent analysis of queueing models and their heavy-traffic limits employing strongly convergent arguments. 


The main goal of this paper is to establish a pathwise \emph{{  conditional} time-homogeneous} approximation scheme to a general finite-state \emph{time-inhomogeneous} MJP $\mathcal{J}$. Our construction has two main components. The first one is a novel modification of the classic uniformization method (\cite{jensen1953markoff,grassmann1977transient,van1992uniformization}), the latter being a simulation method for MJPs which consists in placing their jump times at the arrival epochs of a Poisson process, with the jump mechanism being governed by a discrete-time Markov chain; in the time-inhomogenous case the transitions are dependent on the times at which jumps occur. The key idea behind our scheme consists in placing the path of such discrete-time Markov chain in an \emph{independent} Poisson process of identical intensity: this new process is now a {  conditional} time-homogeneous Markov jump process with infinite state-space, but its jump times no longer match the original ones (thus, can only be considered an approximation). The second component of our construction relies on letting the intensity of the Poisson process associated with the uniformization go to infinity and to establish \emph{strong convergence} of the approximations. The latter is done by measuring the discrepancies between the path of the original time-inhomogeneous MJP and its time-homogeneous approximation in the $J_1$ or Skorokhod sense (see \cite{skorokhod1956limit}). Constructing a strong approximation scheme to time-inhomogeneous MJPs is not only interesting from a theoretical perspective, but it also allows us to infer properties of the \emph{law} of $\mathcal{J}$ by simply appealing to the rule ``strong convergence implies weak convergence''. Thus, one of the byproducts of our strongly convergent result is establishing an alternative way to approximate {   \emph{any} distributional characteristic of $\mathcal{J}$. In particular, we provide a novel approximation scheme to the transition probabilities for time-inhomogeneous MJPs, with some connections to the uniformization methods proposed in \cite{helton1976numerical,van1992uniformization,rindos1995exact,van1998numerical,arns2010numerical}.}

Additionally to the aforementioned theoretical results, our strongly convergent construction sets the stage for a transparent analysis between \emph{functionals} of the process and their respective approximations, and its versatility is one of its strengths. To showcase our results, we study the family of time-inhomogeneous phase-type distributions (\cite{albrecher2019inhomogeneous}), corresponding to the law of the termination time of a time-inhomogeneous MJP. Our convergence results readily translate into providing approximations to time-inhomogeneous phase-type distributions via an infinite-dimensional version of the classic phase-type distribution (\cite{neuts1975probability}), which is much easier to deal with. Indeed, although the former distributions have caught plenty of attention in recent years (starting with \cite{albrecher2019inhomogeneous} and subsequent covariate-dependent and multivariate extensions), it is not surprising that the toolkit available in the literature for phase-type distributions is more robust than its time-inhomogeneous counterpart (see e.g. \cite{bladt2017matrix}) since the product integral is only really computationally inexpensive in the homogeneous case which corresponds to the matrix-exponential. Our approximation yields a way to exploit such a toolkit, with confidence that the approximation converges, and without the need to use stochastic approximation methods such as Markov Chain Monte Carlo (\cite{bladt2003estimation}) or Expectation-Maximization (\cite{asmussen1996fitting}) algorithms. The illustrations we have chosen to provide in this paper are in the domains of univariate data estimation, ruin probabilities, and multivariate distributions.

{\red An interesting related class of distributional approximations for time-homogeneous processes, based on rescaling, is that of uniform acceleration \cite{massey1985asymptotic}. In the context of time-inhomogeneous Markov jump processes, the authors in \cite{massey1998uniform} exploit a high-intensity uniformization technique to provide distributional approximations of the process in terms of its time-varying steady-state properties. Efforts to extend the notion of uniform approximation to a strong pathwise sense have been undertaken in, e.g., \cite{mandelbaum1995strong,mandelbaum1998strong,hampshire2006fluid}. However, these have mainly focused on queueing systems and their functional scaling limits (fluid or diffusive), rather than on time-inhomogeneous Markov jump processes. Further, there are key technical differences that make its analysis considerably different from our current framework (see Remark \ref{rem:ua1} for details). Nonetheless, our work could potentially serve as a stepping stone towards a strong version of the uniform acceleration technique for time-inhomogeneous Markov jump processes. }

This article is structured as follows. In Section \ref{sec:preliminaries} we lay out the components needed in the theory of time-inhomogeneous MJP, as well as on strong convergence of c\`adl\`ag paths. With these tools in hand, in Section \ref{sec:strongMJP} we employ novel techniques to construct a sequence of strongly convergent time-homogeneous {  conditional approximations to any arbitrary time-inhomogeneous MJP, and establish weakly convergent proxies for dealing with the unconditional distributions.} Using the aforementioned results, in Section \ref{sec:IPH} we provide examples where our scheme produces approximations to descriptors of interest in the phase-type literature: univariate hazard-rate density estimation, ruin probability approximations, and multivariate phase-type density evaluation. Finally, in Section \ref{sec:extensions} we conclude with a general discussion on extensions to other jump process.

\section{Preliminaries}\label{sec:preliminaries}
For $p\ge 1$, consider a function $\bm{\Lambda}(t)=\{\Lambda_{ij}(t)\}_{i,j}$, $0\le t <\infty$, which takes values in the space of $p\times p$ real matrices and is such that 
\begin{itemize}
  \item $\bm{\Lambda}(\cdot)$ is Borel measurable and integrable over all compact intervals,
  \item $\Lambda_{ii}(t)\le 0$ and $\Lambda_{ij}(t)\ge 0$ for all $t\ge 0$, $i\neq j\in\{1,\dots,p\}$,
  \item $\sum_j\Lambda_{ij}(t)\le 0$ for all $t\ge 0$, $i\in\{1,\dots,p\}$;
\end{itemize}
we refer to $\bm{\Lambda}(\cdot)$ as an \emph{intensity matrix function}. Denote the \emph{product integral} w.r.t. $\bm{\Lambda}(\cdot)$ over the interval $(s,t)$ by
\begin{equation}\label{eq:prodint1}\prod_s^t (\bm{I} + \bm{\Lambda}(u)\dd u) := \boldsymbol{I}+\sum_{k=1}^{\infty} \int_{s}^{t} \int_{s}^{u_{k}} \cdots \int_{s}^{u_{2}} \mathbf{\Lambda}\left(u_{1}\right) \cdots \mathbf{\Lambda}\left(u_{k}\right) \dd u_{1} \cdots \dd u_{k}.\end{equation}
Under such conditions, it can be shown (see e.g. \cite[Section 4.3]{gill1987product}) that the function
$\bm{P}(s,t)$, $0\le s,t <\infty$, defined by
\[\bm{P}(s,t)=\prod_s^t (\bm{I} + \bm{\Lambda}(u)\dd u)\]
yields a (possibly) defective probability semigroup, in the sense that 
\begin{itemize}
  \item $\bm{P}(s,t)$ is a substochastic matrix,
  \item $\bm{P}(s,t)=\bm{P}(s,u)\bm{P}(u,t)$ for all $0\le s\le u\le t <\infty$,
  \item $\bm{P}(s,s)=\bm{I}$ and $\lim_{t\downarrow s}\bm{P}(s,t)=\bm{I}$.
\end{itemize}
Now, defined on a probability space $(\Omega,\mathbb{P},\mathcal{F})$, let $\mathcal{J}=\{J(t)\}_{t\ge 0}$ be a stochastic process  taking values in $\mathcal{E}=\{1,\dots, p\}$. Assume that each realization of $\mathcal{J}$ is an element of $\mathcal{D}(\mathbb{R}_+, \mathcal{E})$, the space of c\`adl\`ag functions mapping $\mathbb{R}_+$ to $\mathcal{E}$, where $\mathcal{E}$ is endowed with the discrete metric. Furthermore, assume that for all $ 0\le s\le t <\infty,\; i,j\in\mathcal{E}$,
\begin{align}\label{eq:transitionprob1}
\mathbb{P}(J(t)=j\,|\,J(s)=i, \{J(u)\}_{u=0}^s)&=\mathbb{P}(J(t)=j\,|\,J(s)=i)= \left[\bm{P}(s,t)\right]_{ij}.
\end{align}
Then we say that $\mathcal{J}$ is a \emph{time-inhomogeneous Markov jump process} driven by the intensity matrix function $\bm{\Lambda}(\cdot)$. The reason why $\bm{\Lambda}(\cdot)$ is called an ``intensity matrix function'' stems from the fact that if $\Lambda_{ij}(\cdot)$ is right-continuous at $s\ge 0$ for $i,j\in\mathcal{E}$, then
\[\mathbb{P}(J(s+h)=j\,|\,J(s)=i)=\delta_{ij} + \Lambda_{ij}(s)h + o(h),\]
where $\delta_{ij}$ denotes the Kronecker delta function and $o(h)$ an arbitrary function $g:\mathbb{R}_+\rightarrow\mathbb{R}$ such that $\lim_{h\downarrow 0} g(h)/h =0$.

\begin{remark}\label{rem:star1}\rm
Since $\bm{P}(s,t)$ may be defective for some values of $s$ and $t$, the process $\mathcal{J}$ may be terminating. In such a case, we let $\star$ denote an absorbing cemetery state (not included in $\mathcal{E}$) which is entered once/if $\mathcal{J}$ ever gets terminated. If this modification is needed, then the path realizations of $\mathcal{J}$ will be considered to belong to the space $\mathcal{D}(\mathbb{R}_+,\mathcal{E}\cup\{\star\})$ of c\`adl\`ag functions mapping from $\mathbb{R}_+$ to $\mathcal{E}\cup\{\star\}$.
\end{remark}

Given an intensity matrix function $\bm{\Lambda}(\cdot)$, we can always construct an appropriate probability space where a time-inhomogeneous Markov jump process $\mathcal{J}$ driven by $\bm{\Lambda}$ exists \cite[Section 3]{gill1994lectures}. A particularly simple construction through the method of uniformization (see e.g. \cite{van2018uniformization} for a modern exposition) is possible under the assumption that $\bm{\Lambda}(\cdot)$ is uniformly bounded: we will provide an account of such construction in Section~\ref{sec:strongMJP}.

\begin{remark}\rm
The class of time-inhomogeneous Markov jump processes defined here is by no means the most general one. For instance, using the general theory of product integration w.r.t. interval additive functionals (see e.g. \cite{gill1987product}) we can consider semigroups (and in turn time-inhomogeneous Markov jump processes) that allow for jumps to occur at deterministic epochs. This and other constructions will not be discussed further in this manuscript.
\end{remark}

The two main existing techniques to calculate the transition probabilities (\ref{eq:transitionprob1}) rely on computing the product integral (\ref{eq:prodint1}) by either solving an associated ODE \cite[Chapter 1]{Dollard1984ProductIW}, or by approximating it via a discretization argument (\cite{helton1976numerical,van1992uniformization}). If the process $\mathcal{J}$ happens to be time-homogeneous, that is, $\bm{\Lambda}(s)=\bm{\Lambda}$ for all $s\ge 0$, then its transition probabilities (\ref{eq:transitionprob1}) correspond to the $(i,j)$-th entry of
\[\bm{I} + \sum_{k=1}^\infty \frac{\Lambda^k (t-s)^k}{k!}=: e^{\bm{\Lambda}(t-s)}.\]
The matrix-exponential function $x\rightarrow e^{\bm{\Lambda}x}$ is easier to compute and manipulate than the general product integral (\cite{moler2003nineteen}), which is partly the reason why the use of time-homogeneous models in applied probability is more widespread than their time-inhomogeneous counterpart. {\red The emphasis of our construction is thus on simplicity, though this does not necessarily translate into the highest computational efficiency across alternative methods. Instead, }we believe that bridging the theory of time-inhomogeneous Markov jump processes with that of time-homogeneous Markov jump processes via pathwise approximations will improve the applicability of the former in areas where the latter have already been proven to be tractable. For this reason, we now discuss some aspects of the convergence of stochastic processes.

When talking about the convergence of stochastic processes, there are two main notions to be considered: weak and strong. Weak convergence describes how the \emph{laws} associated with the paths of a sequence of processes converge to the law of another process; these processes do not need to exist in the same probability space. On the other hand, strong convergence is interested in measuring the \emph{distance of paths} and guaranteeing that a.s. they converge to each other w.r.t. some metric or topology; these processes need to be defined in a common probability space. As their names indicate, any sequence of stochastic processes which converges strongly will also do it in a weak sense, but the converse is in general not true. 

The main disadvantage when approaching strong convergence of stochastic processes is the lack of general guidelines on how to construct the \emph{appropriate} probability space where convergence occurs: this has to be tailored on a case-by-case basis. In comparison, weak convergence has a well-established set of conditions that have to be verified (for example, tightness of measures and finite-dimensional distributional convergence) which although tedious, do not require any particular probability space construction. Nevertheless, once a strongly convergent approximation is established, its construction is usually transparent, not only providing a more powerful result than its weak counterpart but also bypassing some of its technical difficulties.

Since one of the aims of this paper is to establish strong approximations for some class of c\`adl\`ag processes, first we ought to properly define what \emph{convergence} means in their set of trajectories. It is well-known \cite[Page 289]{jacod2013limit} that assigning a \emph{sensible} topology (and in turn characterizing convergent sequences) to the space $\mathcal{D}(\mathbb{R}_+,\mathcal{E})$ is more involved than, for example, to the space of continuous functions $\mathcal{C}(\mathbb{R}_+,\mathbb{R})$. For instance, if we were to assign to $\mathcal{D}(\mathbb{R}_+,\mathcal{E})$ the topology inherited by the norm
\[\Vert y-z\Vert_\infty =\sum_{k=1}^\infty 2^{-k} \sup_{s\le {  c_k}}\,d_\mathcal{E}(y(s),z(s)),\quad y,z\in \mathcal{D}(\mathbb{R}_+,\mathcal{E}),\]
where $d_\mathcal{E}$ denotes the discrete metric on $\mathcal{E}$ (i.e. $d_{\mathcal{E}}(i,j)= 1-\delta_{ij}$) and {  $\{c_k\}_k$ is an increasing sequence that converges to $\infty$}, then the $2^{-k}$-radius open ball around $y$  would only contain those paths which coincide \emph{exactly} with $y$ in the interval $[0,c_{k+1}]$. While this is a perfectly fine topology, its use is limited by the fact that the only convergent sequences are those whose elements coincide with the limiting element over increasing compact intervals. 

A more \emph{helpful} topology is $J_1$ (also known as Skorokhod topology), which allows for a time distortion to occur as long as this distortion is close to the \emph{real} running time. {  Namely, the $J_1$ topology on $\mathcal{D}(\mathbb{R}_+,\mathcal{E})$ is inherited by the $J_1$ metric which takes the form
\begin{equation}\label{eq:J1met}d_{J_1}(y,z) =\inf_{\Delta \in \mathds{H}}\left\{\left(\sum_{k=1}^\infty 2^{-k} \sup_{s\le c_k}\,d_\mathcal{E}(y(\Delta(s)),z(s))\right)\vee \sup_{s\ge 0}|\Delta(s)-s| \right\},\quad y,z\in \mathcal{D}(\mathbb{R}_+,\mathcal{E}),\end{equation}
where $\mathds{H}$ denotes the set of homeomorphic functions on $\mathbb{R}_+$.
In simple terms, the distance between $y$ and $z$ is strictly less than $2^{-k}$ if and only if there exists an homeomorphic time-change function $\Delta_*$ such that: $y\circ\Delta_*$ restricted to $[0,c_{k+1}]$ is identical to $z$, and $\Delta_*$ is uniformly less than $2^{-k}$ units from the identity function on $\mathbb{R}_+$.}

\begin{remark}\rm
The $J_1$ metric renders the space $\mathcal{D}(\mathbb{R}_+,\mathcal{E})$ complete and separable \cite[Pages 295-298]{jacod2013limit}, thus {  a Polish space}. While other (weaker) topologies for $\mathcal{D}(\mathbb{R}_+,\mathcal{E})$ exist, their use in the literature is not as widespread as $J_1$, mainly because they are more difficult to analyze; see \cite[Chapter 12]{whitt2002stochastic} for alternative topologies  to $J_1$. 
\end{remark}

\section{Strong convergence to time-inhomogeneous Markov jump processes}\label{sec:strongMJP}
As before, let $\mathcal{J}$ denote a (possibly terminating) time-inhomogeneous Markov jump process on a finite state-space $\mathcal{E}$ with initial distribution $\bm{\alpha}=(\alpha_i)_{i\in \mathcal{E}}$ and intensity matrix function $\bm{\Lambda}(t)=\{\Lambda_{ij}(t)\}_{i,j\in\mathcal{E}}$, $0\le t<\infty$. From now on, we assume that the following holds.

\textbf{Condition 1}
The family $\{\bm{\Lambda}(t)\}_{t\ge 0}$ is uniformly bounded, that is
\[\sup_{i\in\mathcal{E}, t\ge 0} |\Lambda_{ii}(t)|<\infty.\]

The goal of this section is to construct a sequence of tractable jump processes $\{\JJ{n}\}_n$ which converges \emph{strongly} to $\mathcal{J}$ as $n\rightarrow\infty$. To establish strong convergence of a sequence of stochastic processes $\{\JJ{n}\}$ to $\mathcal{J}$ we have to:

\begin{enumerate}
	\item Construct a common probability space where $\{\JJ{n}\}_n$ and $\mathcal{J}$ defined,\label{it:strong1}
	\item Guarantee that, as $n\to\infty$, $\JJ{n}$ converges to $\mathcal{J}$ a.s. with respect to the $J_1$ topology on $\mathcal{D}(\mathbb{R}_+,\mathcal{E})$. \label{it:strong2}
\end{enumerate}
Below we provide a detailed description of each one of these steps.

\textbf{Step (\ref{it:strong1}): Construction.} Consider a probability space $(\Omega,\mathcal{F},\mathbb{P})$ which supports the following independent components:
\begin{itemize}
	\item Two independent Poisson processes $\mathcal{N}=\{N(t)\}_{t\ge 0}$ and $\mathcal{M}=\{M(t)\}_{t\ge 0}$ with common intensity $\lambda_0 := \lceil\sup_{i\in\mathcal{E}, t\ge 0} |\Lambda_{ii}(t)|\rceil$;
	\item Two independent sequences of independent Poisson processes $\{\widehat{\mathcal{N}}^{(k)}\}_{k=\lambda_0+1}^\infty$ and $\{\widehat{\mathcal{M}}^{(k)}\}_{k=\lambda_0+1}^\infty$ of parameter $1$;
   \item A sequence of independent and identically distributed $\mbox{U}(0,1)$ random variables $\{U_\ell\}_{\ell=0}^\infty$.
\end{itemize}
Such a space can always be constructed by Kolmogorov's extension theorem.
With these elements at hand, we construct the processes $\mathcal{J}$ and $\JJ{n}$, $n\ge \lambda_0$ via a uniformization argument as follows. 

First, let $J(0)=i$ if $U_0\in[\sum_{j\le i}\alpha_j, \sum_{j\le i+1}\alpha_j)$, which guarantees that $J(0)\sim \bm{\alpha}$. Now, let $\mathcal{J}$ be constant between the arrival times $\{\chi_\ell\}$ of $\mathcal{N}$ (with $\chi_0:=0$), meaning that
\[J(t)=J\left(\chi_\ell\right)\quad \mbox{for}\quad t\in\left[\chi_\ell, \chi_{\ell+1}\right), \,\ell\ge 0.\]
The epochs $\{\chi_\ell\}_{\ell\ge 1}$ will serve as the (possible) jumping times of $\mathcal{J}$: conditional on the event $\{J\big(\chi_{\ell}-\big)=i\}$, $i\in\mathcal{E}$,
\begin{align}
J(\chi_\ell)=\left\{\begin{array}{ccc}k &\mbox{if}&   U_\ell\in\left[\sum_{j\le k}\lambda^{-1}_0\Lambda_{ij}(\chi_{\ell}) +\delta_{ij}, \sum_{j\le k+1}\lambda^{-1}_0\Lambda_{ij}(\chi_{\ell}) +\delta_{ij}\right),\\
\mbox{terminated}&\mbox{if}& U_\ell\in\left[1 + \sum_{j\in\mathcal{E}}\lambda^{-1}_0\Lambda_{ij}(\chi_{\ell}), 1\right);\end{array}\right.\label{eq:Jtauell1}
\end{align}
if $\mathcal{J}$ is terminated at time $\chi_{\ell}$, it remains terminated forever. Equation (\ref{eq:Jtauell1}) implies that, conditional on $\{\chi_{\ell}\}_{\ell\ge 0}$, the process $\{J(\chi_\ell)\}_{\ell\ge 0}$ is a (possibly terminating) Markov chain with time-inhomogeneous transition probability matrices $\lambda_0^{-1}\bm{\Lambda}(\chi_\ell) + \bm{I}$, $\ell\ge 1$. This construction corresponds to the classic uniformization of a Markov jump process $\mathcal{J}$ with time-inhomogeneous intensities $\{\bm{\Lambda}(t)\}_{t\ge 0}$ and is widely used, e.g. for simulation purposes (see e.g. \cite{van2018uniformization}). 

To construct the approximating sequence $\JJ{n}$, $n\ge \lambda_0$, the idea is to inspect $\mathcal{J}$ at an increasingly denser set of Poisson epochs, and place such observations at the arrival epochs of an independent Poisson process with identical distibutional characteristics. To this end, for $n > \lambda_0$, let $\mathcal{N}^{(n)}$ be the superposition of $\mathcal{N}$ and $\widehat{\mathcal{N}}^{(\lambda_0+1)},\dots, \widehat{\mathcal{N}}^{(n)}$; $\mathcal{N}^{(n)}$ is then a Poisson process of parameter $n$, which we will use as ``inspection'' epochs of $\mathcal{J}$. Similarly, let $\mathcal{M}^{(n)}$ be the superposition of $\mathcal{M}$ and $\widehat{\mathcal{M}}^{(\lambda_0+1)},\dots, \widehat{\mathcal{M}}^{(n)}$; the Poisson process $\mathcal{M}^{(n)}$ will serve as the random grid of $\JJ{n}$ where we will place the observations of $\mathcal{J}$. More specifically, we let
\begin{equation}\label{eq:defJnt}J^{(n)}(t)= J\left(\chi^{(n)}_{\ell}\right)\quad\mbox{for}\quad t\in \left[\theta_{\ell}^{(n)}, \theta_{\ell+1}^{(n)}\right),\end{equation}
where $\big\{\chi_{\ell}^{(n)}\big\}_{\ell\ge 0}$ and $\big\{\theta^{(n)}_{\ell}\big\}_{\ell\ge 0}$ correspond to the arrival times of $\mathcal{N}^{(n)}$ and $\mathcal{M}^{(n)}$, respectively. Below we show how $\JJ{n}$, {   conditional on $\mathcal{N}^{(n)}$}, can be embedded into a time-homogeneous Markov jump process.

\begin{theorem}\label{th:approximation1}
Fix $n\ge \lambda_0$ and let $\mathcal{J}$ and $\JJ{n}$ be the above processes constructed in $(\Omega,\mathcal{F},\mathbb{P})$. Then the following statements hold.
\begin{enumerate} 
   \item[(a)] $\big\{J\big(\chi^{(n)}_{\ell}\big)\big\}_{\ell\ge 0}$ {   conditional on $\mathcal{N}^{(n)}$} is an inhomogeneous discrete-time Markov chain with initial distribution $\bm{\alpha}$ and transition probability matrix at the $\ell$-th step given by
\begin{equation}\label{eq:Qnell1}
{  \bm{Q}_\ell^{(n)}=\bm{I} + \frac{\bm{\Lambda}(\chi_{\ell}^{(n)})}{n}.}\end{equation}
\item[(b)] The bivariate process $\big\{\big(L^{(n)}(t), J^{(n)}(t)\big)\big\}_{t\ge 0}$ given by
\begin{equation}\label{eq:bivariate1}\left(L^{(n)}(t), J^{(n)}(t)\right)= \left(\ell, J\left(\chi^{(n)}_{\ell}\right)\right)\quad\mbox{for}\quad t\in \left[\theta^{(n)}_{\ell}, \theta^{(n)}_{\ell+1}\right),\end{equation}
{   conditional on $\mathcal{N}^{(n)}$}, is a time-homogeneous Markov jump process with initial distribution $(\bm{\alpha},\bm{0},\bm{0},\dots)$ and intensity matrix given by
\begin{equation}\label{eq:intensityinfinite1}\begin{pmatrix}-n\bm{I}&n\bm{Q}^{(n)}_1& &&\\
&-n\bm{I}&n\bm{Q}^{(n)}_2& &\\
&&-n\bm{I}&n\bm{Q}^{(n)}_3&\\
&&&\ddots&\ddots
\end{pmatrix},\end{equation}
where the state-space $\mathbb{N}_0\times\mathcal{E}$ is to be understood in a lexicographic sense.
\end{enumerate}
\end{theorem}
\begin{proof}\textbf{Proof.}

\begin{enumerate}
  \item[(a)] Let $i,j\in\mathcal{E}$, $\ell\ge 1$ 
 {   and denote by $\Pg$ the probability measure $\P$ conditioned on the $\sigma$-algebra $\mathcal{G}_n$ generated by $\mathcal{N}^{(n)}=\big\{\chi_{\ell}^{(n)}\big\}_{\ell\ge 0}$.} Then,
  \begin{align*}
   {  \Pg}&\left( J(\chi^{(n)}_{\ell})=j \mid J(\chi^{(n)}_{\ell-1})=i\right)\\
   & =  {  \Pg}\left( J(\chi^{(n)}_{\ell})=j, {  N(\chi^{(n)}_{\ell}-)\neq N(\chi^{(n)}_{\ell})} \mid J(\chi^{(n)}_{\ell}-)=i\right)\\
   &\quad + {  \Pg}\left( J(\chi^{(n)}_{\ell})=j, {  N(\chi^{(n)}_{\ell}-) = N(\chi^{(n)}_{\ell})} \mid J({  \chi^{(n)}_{\ell}-})=i\right)\\
   & = {  \left(\frac{\lambda_0}{n}\right) \left(\lambda^{-1}_0\Lambda_{ij}(\chi^{(n)}_{\ell}) + \delta_{ij}\right) + \left(\frac{n-\lambda_0}{n}\right)\delta_{ij}}\\
   & = {   n^{-1}  \Lambda_{ij}(\chi^{(n)}_{\ell}) + \delta_{ij}},
   \end{align*}
   from which (\ref{eq:Qnell1}) follows. Also, recall that $J(\chi^{(n)}_0) = J(0)\sim\bm{\alpha}$ by construction. 
   \item[(b)] It is standard that the enlarged process $\big\{\big(\ell, J\big(\chi^{(n)}_{\ell}\big)\big)\big\}_{\ell\ge 0}$, {   conditional on $\mathcal{N}^{(n)}$}, is a time-homogeneous Markov chain with initial distribution $(\bm{\alpha},\bm{0},\bm{0},\dots)$ and transition probability matrix given by
   \[\begin{pmatrix}\bm{0}&\bm{Q}^{(n)}_1& &&\\
&\bm{0}&\bm{Q}^{(n)}_2& &\\
&&\bm{0}&\bm{Q}^{(n)}_3&\\
&&&\ddots&\ddots
\end{pmatrix}.\]
Furthermore, since $\big\{\big(\ell, J\big(\chi^{(n)}_{\ell}\big)\big)\big\}_{\ell\ge 0}$ is independent of $\mathcal{M}^{(n)}$ by construction, then $\big\{\big(L^{(n)}(t), J^{(n)}(t)\big)\big\}_{t\ge 0}$ as defined in (\ref{eq:bivariate1}), {   when conditioned on $\mathcal{N}^{(n)}$}, corresponds to a uniformized Markov jump process with intensity matrix of the form (\ref{eq:intensityinfinite1}) (see for instance \cite[Section 3.2]{van2018uniformization}).
\end{enumerate}
$\hfill\square$\end{proof}
In conclusion, Theorem \ref{th:approximation1} implies that the process $\JJ{n}$ {   conditional on $\mathcal{N}^{(n)}$} is simply a projection of a time-homogeneous Markov jump process described by (\ref{eq:intensityinfinite1}).
{  
\begin{theorem}\label{th:transition_prob1}
The transition probabilities of $\mathcal{J}^{(n)}$ conditional on $\mathcal{N}^{(n)}$, from time $s$ to $t$ ($t\ge s\ge 0$), are given by
\begin{align}\label{eq:ftaun}
\bm{P}^{(n)}(s,t)&=\sum_{\ell=0}^\infty \sum_{k=0}^\infty \mathrm{Poi}_{ns}(k)\times \mathrm{Poi}_{n(t-s)}(\ell)\times \left[\bm{Q}^{(n)}_{k+1}\bm{Q}^{(n)}_{k+2}\cdots\bm{Q}^{(n)}_{k+\ell}\right],
\end{align}
with {\red $\mathrm{Poi}_{\lambda}(m)=\tfrac{\lambda^m}{m!} e^{-\lambda}$}. Here, empty matrix products are defined as the identity matrix.
\end{theorem}
\begin{proof}\textbf{Proof.}
Using the notation in the proof of Theorem \ref{th:approximation1}, we have that 
\begin{align*}
\left(\bm{P}^{(n)}(s,t)\right)_{ij} & = \Pg (J^{(n)}(t)=j\mid J^{(n)}(s)=i)\\
&=\sum_{k=0}^\infty \sum_{\ell=0}^\infty \Pg (J^{(n)}(t)=j, M^{(n)}(s) = k, {\red M^{(n)}(t) - M^{(n)}(s) = \ell}\mid J^{(n)}(s)=i)\\
&=\sum_{k=0}^\infty \sum_{\ell=0}^\infty \mathrm{Poi}_{ns}(k)\times \mathrm{Poi}_{n(t-s)}(\ell)\times \Pg (J (\chi^{(n)}_{k+\ell})=j\mid J (\chi^{(n)}_{k}) = i)\\
& = \sum_{\ell=0}^\infty \sum_{k=0}^\infty \mathrm{Poi}_{ns}(k)\times \mathrm{Poi}_{n(t-s)}(\ell)\times \left[\bm{Q}^{(n)}_{k+1}\bm{Q}^{(n)}_{k+2}\cdots\bm{Q}^{(n)}_{k+\ell}\right]_{ij},
\end{align*}
from which the result follows.
$\hfill\square$\end{proof}}

\textbf{Step (\ref{it:strong2}): Convergence.} Following the reasoning in Remark \ref{rem:star1}, let us work in the space $\mathcal{D}(\mathbb{R}_+,\mathcal{E}_\star)$ (where $\mathcal{E}_\star=\mathcal{E}\cup\{\star\}$) instead of $\mathcal{D}(\mathbb{R}_+,\mathcal{E})$, with $\star$ denoting the terminating state. If termination never happens, the extra state $\star$ will never get entered. Endow $\mathcal{D}(\mathbb{R}_+,\mathcal{E}_\star)$ with the $J_1$ topology which is inherited by (\ref{eq:J1met}).

To show convergence of $\JJ{n}$ to $\mathcal{J}$ it is sufficient to exhibit a sequence of random homeomorphic functions $\{\Delta_n\}_n$ such that for each $\omega\in\Omega$ (more precisely for all $\omega$'s in some set of measure $1$),
\begin{align}\sup_{s\in[0,K]}d_{\mathcal{E}_\star} \left(J^{(n)}\big(\Delta_n(s)\big)(\omega),\, J(s)(\omega)\right)\rightarrow 0\quad\mbox{as}\quad n\rightarrow {  \infty}\quad\mbox{for all}\quad K\in \mathbb{N}_0,\label{eq:strongJ5}
\end{align}
\begin{align}
\sup_{s\ge 0}|\Delta_n(s)(\omega)-s|\rightarrow 0\quad\mbox{as}\quad n\rightarrow {  \infty}.\label{eq:strongJ6}
\end{align}
Here we define 
\begin{align}
\Delta_n(t)=\left\{\begin{array}{ccc}  t\frac{\theta_1^{(n)}}{\chi_1^{(n)}}&\mbox{for}& t\in \left[0,\chi_1^{(n)}\right)\\
\theta^{(n)}_\ell + (t-\chi^{(n)}_\ell)\frac{\theta_{\ell+1}^{(n)}-\theta_{\ell}^{(n)}}{\chi_{\ell+1}^{(n)}-\chi_{\ell}^{(n)}}&\mbox{for}& t\in \left[\chi_\ell^{(n)}, \chi_{\ell+1}^{(n)}\right),\,\ell<\lfloor n^{1+\varepsilon}\rfloor\\
\theta^{(n)}_{\lfloor n^{1+\varepsilon}\rfloor} + (t-\chi^{(n)}_{\lfloor n^{1+\varepsilon}\rfloor})&\mbox{for}& t\in \left[\chi_{\lfloor n^{1+\varepsilon}\rfloor}^{(n)}, \infty\right).
\end{array}\right.\label{eq:defDelta1}
\end{align}
where $\varepsilon>0$ is arbitrary but fixed. The function $\Delta_n(\cdot)$ in (\ref{eq:defDelta1}) is a.s. strictly monotone and continuous with $\Delta_n(0)=0$ and $\lim_{t\rightarrow\infty}\Delta_n(t)=\infty$, thus homeomorphic. Additionally, $\Delta_n(\cdot)$ is piecewise linear, completely characterized by the values $\Delta_n(\chi_\ell^{(n)})=\theta_\ell^{(n)}$, $\ell=1,\dots,\lfloor n^{1+\varepsilon}\rfloor$, and the slope $1$ after time $\chi_{\lfloor n^{1+\varepsilon}\rfloor}^{(n)}$. Thus,
\begin{align}
J^{(n)}(\Delta_n(s))= J(s)\quad\mbox{for all}\quad s\in\left[0,\chi^{(n)}_{\lfloor n^{1+\varepsilon}\rfloor}\right),\quad\mbox{and}\label{eq:convaux1}\\
\sup_{s\ge 0}|\Delta_n(s)-s| =\max_{\ell\in\{1,\dots, \lfloor n^{1+\varepsilon}\rfloor\}}\left|\theta_\ell^{(n)}-\chi_\ell^{(n)}\right|.\label{eq:convaux2}\end{align}
Equations (\ref{eq:convaux1}) and (\ref{eq:convaux2}) together with the following lemma will help us establish the desired strong convergence of $\JJ{n}$ to $\mathcal{J}$ in Theorem \ref{th:strongmain1} below.
\begin{lemma}\label{lem:ratestrong1}
For all $q> 0$ and $\varepsilon\in (0,1)$ there exists some $\kappa(q,\varepsilon )>0$ such that
\begin{align}
\mathbb{P}\left(\max_{\ell\in\{1,\dots, \lfloor n^{1+\varepsilon}\rfloor\}}\left|\theta_\ell^{(n)}-\chi_\ell^{(n)}\right|\ge \kappa(q,\varepsilon)(\log n)n^{-1/2+{  \varepsilon/2}}\right)=o(n^{-q}).\label{eq:ratestrong1}
\end{align}
\end{lemma}
\begin{proof}\textbf{Proof.}
Let $\{\beta_n\}_{n\ge 1}$ denote an arbitrary positive sequence. Notice that $\mathbb{E}(\theta_\ell^{(n)})=\ell/n$, so that Doob's $L_p$-maximal inequality implies
\begin{align*}
\mathbb{P}\left(\max_{\ell\in\{1,\dots, \lfloor n^{1+\varepsilon}\rfloor\}}\left|\theta_\ell^{(n)}-\ell/n\right|\ge \beta_n\right)\le \frac{\mathbb{E}\left(\left(\theta_{\lfloor n^{1+\varepsilon}\rfloor}^{(n)}-\lfloor n^{1+\varepsilon}\rfloor/n)\right)^{2k}\right)}{\beta_n^{2k}}\quad\mbox{for all}\quad k\in\mathbb{N}_0.
\end{align*}
Using Lemma A.1. in \cite{nguyen2022rate} {  for $n\ge 4$}, we may bound Erlang central moments as follows:
\begin{align*}
\mathbb{E}\left(\left(\theta_{\lfloor n^{1+\varepsilon}\rfloor}^{(n)}-\lfloor n^{1+\varepsilon}\rfloor/n)\right)^{2k}\right)&\le \frac{(2k)!\sqrt{\lfloor n^{1+\varepsilon}\rfloor}}{n^{2k}}\frac{\sqrt{\lfloor n^{1+\varepsilon}\rfloor}^{2k+1}-1}{\sqrt{\lfloor n^{1+\varepsilon}\rfloor}-1}\le {  2}\sqrt{\lfloor n^{1+\varepsilon}\rfloor}\left(\frac{2k\sqrt{\lfloor n^{1+\varepsilon}\rfloor}}{n}\right)^{2k},
\end{align*}
{  where in the second inequality we used that $(2k)!\le (2k)^{2k}$ and $\tfrac{1}{\sqrt{\lfloor n^{1+\varepsilon}\rfloor} - 1}\le \tfrac{2}{\sqrt{\lfloor n^{1+\varepsilon}\rfloor}}$ for $n\ge 4$.}
Now, choosing $k=\lfloor \log n\rfloor$,
\begin{align*}
\mathbb{E}\left(\left(\theta_{\lfloor n^{1+\varepsilon}\rfloor}^{(n)}-\lfloor n^{1+\varepsilon}\rfloor/n)\right)^{2\lfloor \log n\rfloor}\right)&\le {  2}\sqrt{\lfloor n^{1+\varepsilon}\rfloor}\left(\frac{2\lfloor \log n\rfloor\sqrt{\lfloor n^{1+\varepsilon}\rfloor}}{n}\right)^{2\lfloor \log n\rfloor}\\
&\le {  2}\sqrt{n^{1+\varepsilon}}\left(2 (\log n)  n^{-1/2+\varepsilon/2}\right)^{2\lfloor \log n\rfloor},
\end{align*}
so that for $\beta_n={  2}\, C(\log n)n^{-1/2+\varepsilon/2}$ with {  $C\ge 1$ and $n\ge e^2$},
{  \begin{align}
\mathbb{P}\left(\max_{\ell\in\{1,\dots, \lfloor n^{1+\varepsilon}\rfloor\}}\left|\theta_\ell^{(n)}-\ell/n\right|\ge \beta_n\right) & \le \frac{{  2}\sqrt{n^{1+\varepsilon}}\left(2 (\log n)  n^{-1/2+\varepsilon/2}\right)^{2\lfloor \log n\rfloor}}{\left({  2} C(\log n)n^{-1/2+\varepsilon/2}\right)^{2\lfloor \log n\rfloor}}\label{eq:beta1}\\
& = \frac{{  2}\sqrt{n^{1+\varepsilon}}}{C^{2 \lfloor \log n\rfloor}}\le \frac{{  2}\sqrt{n^{1+\varepsilon}}}{C^{2 (\log n-1)}} \le \frac{{  2}\sqrt{n^{1+\varepsilon}}}{C^{\log n}}.\nonumber
\end{align} }
Note that choosing {  $C=e^{1/2+\varepsilon/2+2q}$} yields {  $2\sqrt{n^{1+\varepsilon}}\, C^{-\log n}=2n^{-2q}$}, rendering the l.h.s. of (\ref{eq:beta1}) an $o(n^{-q})$ function.
Since
\begin{equation}\label{eq:rateaux2}\mathbb{P}\left(\max_{\ell\in\{1,\dots, \lfloor n^{1+\varepsilon}\rfloor\}}\left|\theta_\ell^{(n)}-\ell/n\right|\ge \beta_n\right)=\mathbb{P}\left(\max_{\ell\in\{1,\dots, \lfloor n^{1+\varepsilon}\rfloor\}}\left|\chi_\ell^{(n)}-\ell/n\right|\ge \beta_n\right)=o(n^{-q})\quad\mbox{and}\end{equation}
\begin{align*}
\mathbb{P}\left(\max_{\ell\in\{1,\dots, \lfloor n^{1+\varepsilon}\rfloor\}}\left|\theta_\ell^{(n)}-\chi_\ell^{(n)}\right|\ge 2\beta_n\right)&\le \mathbb{P}\left(\max_{\ell\in\{1,\dots, \lfloor n^{1+\varepsilon}\rfloor\}}\left|\theta_\ell^{(n)}-\ell/n\right|\ge \beta_n\right)\\
&\quad + \mathbb{P}\left(\max_{\ell\in\{1,\dots, \lfloor n^{1+\varepsilon}\rfloor\}}\left|\chi_\ell^{(n)}-\ell/n\right|\ge \beta_n\right),
\end{align*}
then (\ref{eq:ratestrong1}) follows with $\kappa(q,\varepsilon )={  4 e^{1/2+\varepsilon/2+2q}}$.
$\hfill\square$\end{proof}
{ 
\begin{theorem}\label{th:strongmain1}
For all $q> 0$ and $\varepsilon\in (0,1)$ there exists some nonnegative sequence $\{c_k\}_k$ converging to $\infty$ such that
\begin{align}\label{eq:ratemainsquared}
\mathbb{P}\left({\red d_{J_1}(\mathcal{J}^{(n)}, \mathcal{J})}\ge 3\kappa(q,\varepsilon)(\log n)n^{-1/2+{  \varepsilon/2}}\right)=o(n^{-q}),
\end{align}
where $d_{J_1}$ depends on $\{c_k\}_k$ through (\ref{eq:J1met}). In particular, $\mathcal{J}^{(n)}$ converges strongly to $\mathcal{J}$ as $n\rightarrow\infty$ in the $J_1$ topology over $\mathcal{D}(\mathbb{R}_+,\mathcal{E}_\star)$.
\end{theorem}
\begin{proof}\textbf{Proof.}
Take $c_n=\lfloor n^{1+\varepsilon}\rfloor /n - \beta_n$ where $\beta_n=2 e^{1/2+\varepsilon/2+2q} (\log n)n^{-1/2+\varepsilon/2}$; this particular choice for $\{c_n\}_n$ will become apparent later in (\ref{eq:eventsaux1}). Notice that $\{c_n\}_n$ may not be increasing, but eventually increasing. Now, for any $m\ge 1$ such that $c_m\ge \max\{c_1,\cdots,c_m\}$,
\begin{align*}d_{J_1}(\mathcal{J}^{(n)},\mathcal{J}) &\le \left(\sum_{k=1}^\infty 2^{-k} \sup_{s\le c_k}\,d_\mathcal{E}(J^{(n)}(\Delta_n(s)),J(s))\right)\vee \sup_{s\ge 0}|\Delta_n(s)-s|\\
& \le \left(\sup_{s\le c_m}\,d_\mathcal{E}(J^{(n)}(\Delta_n(s)),J(s)) + 2^{-m}\right) + \sup_{s\ge 0}|\Delta_n(s)-s|. \end{align*}
Taking $m=n$ for large enough $n$, we get
\begin{align*}
\left\{ {\red d_{J_1}(\mathcal{J}^{(n)}, \mathcal{J})}\ge 3\kappa(q,\varepsilon)(\log n)n^{-1/2+{  \varepsilon/2}}\right\}&\subseteq A^{(n)}_1\cup A^{(n)}_2\cup A^{(n)}_3
\end{align*}
where
\begin{align*}
A^{(n)}_1 & = \left\{ \sup_{s\le c_n}\,d_\mathcal{E}(J^{(n)}(\Delta_n(s)),J(s))\ge \kappa(q,\varepsilon)(\log n)n^{-1/2+{  \varepsilon/2}}\right\},\\
A^{(n)}_2 & = \left\{2^{-n}\ge \kappa(q,\varepsilon)(\log n)n^{-1/2+{  \varepsilon/2}}\right\},\\
A^{(n)}_3 & = \left\{\sup_{s\ge 0}|\Delta_n(s)-s|\ge \kappa(q,\varepsilon)(\log n)n^{-1/2+{  \varepsilon/2}}\right\}.
\end{align*}
The event $A^{(n)}_1$ is contained in
\begin{equation}\label{eq:eventsaux1}\left\{J^{(n)}(\Delta_n(s))\neq J(s)\mbox{ for some }s<\chi^{(n)}_{\lfloor n^{1+\varepsilon}\rfloor}\right\}\cup\left\{\chi^{(n)}_{\lfloor n^{1+\varepsilon}\rfloor}\le c_n\right\};\end{equation}
the first set in (\ref{eq:eventsaux1}) is $\mathbb{P}$-null by construction, while the second set is contained in $\{\max_{\ell\in\{1,\dots, \lfloor n^{1+\varepsilon}\rfloor\}}\left|\theta_\ell^{(n)}-\ell/n\right|\ge \beta_n\}$ which according to (\ref{eq:rateaux2}) has an $o(n^{-q})$ probability of occuring. Finally, the set $A^{(n)}_2$ is clearly null for large enough $n$, and $A^{(n)}_3$ has an $o(n^{-q})$ probability in view of (\ref{eq:convaux2}) and Lemma \ref{lem:ratestrong1}. This proves (\ref{eq:ratemainsquared}), with the strong convergence result following from standard Borel-Cantelli arguments.
$\hfill\square$\end{proof}}

{  Not surprisingly, the strongly convergent result in Theorem \ref{th:strongmain1} yields a distributional property of the approximation $\mathcal{J}^{(n)}$: its associated probability law over $\mathcal{D}(\mathbb{R}_+,\mathcal{E}_\star)$ converges to that of $\mathcal{J}$ as $n\rightarrow\infty$. More specifically, for all $\nu > 0$ and subset $A$ of the Polish space $\mathcal{D}(\mathbb{R}_+,\mathcal{E}_\star)$, define the $\nu$-neighbourhood of $A$ by
\[A^{\nu}:=\left\{y\in \mathcal{D}(\mathbb{R}_+,\mathcal{E}_\star) \,:\, \exists z\in A, d_{J_1}(y,z)<\nu\right\}.\]
Then, the Prokhorov metric between two probability measures $\mu_1,\mu_2$ over $\mathcal{D}(\mathbb{R}_+,\mathcal{E}_\star)$ (endowed with its Borel sets, $\mathbb{B}_\mathcal{D}$) is defined by
\[L(\mu_1,\mu_2):=\inf\{\nu > 0\, :\, \mu_1(A)\le \mu_2(A^\nu) + \nu\mbox{ and }\mu_2(A)\le \mu_1(A^\nu) + \nu\mbox{ for all }A\in \mathbb{B}_\mathcal{D}\}.\]
Employing \cite[Lemma 1.2]{prokhorov1956convergence} in a spirit similar to \cite{fraser1973rate}, we readily get the following rate of distributional convergence as a sraightforward consequence of Theorem \ref{th:strongmain1}.
\begin{corollary}\label{cor:prok1} Let $\mu: \mathbb{B}_\mathcal{D}\mapsto[0,1]$ and $\mu^{(n)}:\mathbb{B}_\mathcal{D}\mapsto[0,1]$ be the probability laws associated to $\mathcal{J}$ and $\mathcal{J}^{(n)}$, respectively. Then,
\[L(\mu^{(n)},\mu)=o((\log n)n^{-1/2+{  \varepsilon/2}}).\]
\end{corollary}
}

In summary, Theorems \ref{th:approximation1} and \ref{th:strongmain1} yield a way to study any arbitrary finite-state time-inhomogeneous Markov jump process, under the transparent Condition 1, by means of {   conditional} time-homogeneous Markov jump processes, the latter having an infinite-dimensional block structure. {  Moreover, Corollary \ref{cor:prok1} provides a robust setup to quantify the distributional convergence of $\mathcal{J}^{(n)}$ to $\mathcal{J}$ not only from a pointwise perspective but from a functional one. To the best of our knowledge, this is the first attempt in the literature that provides a functional approach to quantify an approximative uniformization algorithm. Related methods such as those in \cite{helton1976numerical,van1992uniformization,arns2010numerical} work under certain continuity assumptions on $\bm{\Lambda}(\cdot)$ and focus on transition probabilities. In contrast, the present treatment examines the conditional case under fairly general conditions (measurability and boundedness) and provides stronger pathwise results. Next, we study the unconditional case, for which we will require similar regularity conditions to those found in the aforementioned approximative uniformization algorithms.}

{  
\subsection{From conditional to unconditional convergence} \label{sec:uncond1}
The above strongly convergent conditional representations suggest using Monte Carlo simulation methods to obtain accurate unconditional distributional properties of the main approximations. More specifically: 
\begin{enumerate}
\item[(A)] sample the Poisson process $\big\{\chi_{\ell}^{(n)}\big\}_{\ell\ge 0}$, 
\item[(B)] apply closed form formulae to descriptors of the conditional time-homogeneous Markov jump process $\big\{\big(L^{(n)}(t), J^{(n)}(t)\big)\big\}_{t\ge 0}$, 
\item[(C)] repeat steps (A) and (B) a predetermined amount of times and average the resulting descriptors.
\end{enumerate}
While this is procedure is  inexpensive, we suggest using a nonrandom modification of the approximating distributional characteristics. Under certain regularity conditions, we now show that these modifications also converge weakly to the target law. Namely, we require the following \textit{Lipschitz continuity} assumption on the intensity matrix function $\bm{\Lambda}(\cdot)$: there exists some $K_L>0$ such that for all $x,y\in\mathbb{R}$ we have
\begin{equation}\label{lip_cond}\Vert \bm{\Lambda}(x)-\bm{\Lambda}(y)\Vert \le K_L|x-y|,\end{equation}
where $\Vert \cdot \Vert$ in the l.h.s. of \eqref{lip_cond} denotes the supremum norm of matrices.

We define the following (nonrandom) modifications for the components of the transition matrix:
\begin{align}\label{nonrandom_Qs}
\widehat{\bm{Q}}^{(n)}_\ell=\bm{I}+\frac 1n \left[ \bm{\Lambda}(\ell/n) \right],\quad \widetilde{\bm{Q}}^{(n)}_\ell=\bm{I}+\frac 1n \E\left[ \bm{\Lambda}(\chi^{(n)}_\ell) \right],
\end{align}
and denote their corresponding transition matrices by $\widehat{\bm{P}}^{(n)}$ and $\widetilde{\bm{P}}^{(n)}$ where for $0\le s\le t$,
\begin{align}
\widehat{\bm{P}}^{(n)}(s,t)&=\sum_{\ell=0}^\infty \sum_{k=0}^\infty \mathrm{Poi}_{ns}(k)\times \mathrm{Poi}_{n(t-s)}(\ell)\times \left[\widehat{\bm{Q}}^{(n)}_{k+1}\widehat{\bm{Q}}^{(n)}_{k+2}\cdots\widehat{\bm{Q}}^{(n)}_{k+\ell}\right],\label{eq:Phat7}\\
\widetilde{\bm{P}}^{(n)}(s,t)&=\sum_{\ell=0}^\infty \sum_{k=0}^\infty \mathrm{Poi}_{ns}(k)\times \mathrm{Poi}_{n(t-s)}(\ell)\times \left[\widetilde{\bm{Q}}^{(n)}_{k+1}\widetilde{\bm{Q}}^{(n)}_{k+2}\cdots\widetilde{\bm{Q}}^{(n)}_{k+\ell}\right].\label{eq:Ptilde7}
\end{align}
Note that while ${\bm{P}}^{(n)}$ conditional on $\big\{\chi_{\ell}^{(n)}\big\}_{\ell\ge 0}$ is random, the transition probabilities $\widehat{\bm{P}}^{(n)}$ and $\widetilde{\bm{P}}^{(n)}$ are deterministic.

\begin{theorem}\label{gen_convs}
For any $q>1$ and $\varepsilon>0$, there exists a constant $\kappa_L(q,\varepsilon)$ such that for all $T\ge 0$
\begin{align}
\mathbb{P}\left(\sup_{0\le s<t\le T}\Vert \bm{P}^{(n)}(s,t)-\widehat{\bm{P}}^{(n)}(s,t)\Vert\ge (T+1) \kappa_L(q,\varepsilon)(\log n)n^{-1/2+\varepsilon}\right)=o(n^{-q}),\label{eq:ratePhat}\\ 
\mathbb{P}\left(\sup_{0\le s<t\le T}\Vert \bm{P}^{(n)}(s,t)-\widetilde{\bm{P}}^{(n)}(s,t)\Vert\ge (T+1) \kappa_L(q,\varepsilon)(\log n)n^{-1/2+\varepsilon}\right)=o(n^{-q}).\label{eq:ratePtilde}\end{align}
\end{theorem}

\begin{proof}\textbf{Proof.} Let $B^{(n)}=\left\{\max_{\ell\in\{1,\dots, \lfloor n^{1+\varepsilon}\rfloor\}}\left|\theta_\ell^{(n)}-\chi_\ell^{(n)}\right|< \kappa(q,\varepsilon)(\log n)n^{-1/2+\varepsilon}\right\}$; by Lemma \ref{lem:ratestrong1}, $\P(\Omega\setminus B^{(n)})= o(n^{-q})$, so that we can focus our attention on events that occur under $B^{(n)}$ only. Employing the triangle inequality of $\Vert\cdot\Vert$, for $0\le s<t\le T$ 

\begin{align}
\Vert \bm{P}^{(n)}(s,t)-\widehat{\bm{P}}^{(n)}(s,t)& \Vert\le  \sum_{\substack{k,\ell \ge 0\\ k+\ell \le \lfloor n^{1+\varepsilon}\rfloor }} \mathrm{Poi}_{ns}(k)\times \mathrm{Poi}_{n(t-s)}(\ell)\times\left\Vert\prod_{m=1}^{\ell} \bm{Q}^{(n)}_{k+m} - \prod_{m=1}^{\ell} \widehat{\bm{Q}}^{(n)}_{k+m} \right\Vert\nonumber\\
&\quad + \sum_{\substack{k,\ell \ge 0\\ k+\ell > \lfloor n^{1+\varepsilon}\rfloor }} \mathrm{Poi}_{ns}(k)\times \mathrm{Poi}_{n(t-s)}(\ell).\label{eq:triangle1}
\end{align}
Further,
\begin{align*}
\left\Vert \prod_{m=1}^{\ell} \bm{Q}^{(n)}_{k+m} - \prod_{m=1}^{\ell} \widehat{\bm{Q}}^{(n)}_{k+m} \right\Vert&=\left\Vert
\bm{Q}^{(n)}_{k+1}\left(\prod_{m=2}^{\ell} \bm{Q}^{(n)}_{k+m} - \prod_{m=2}^{\ell} \widehat{\bm{Q}}^{(n)}_{k+m} \right)
-\left(\widehat{\bm{Q}}^{(n)}_{k+1}-\bm{Q}^{(n)}_{k+1} \right)\prod_{m=2}^{\ell}\widehat{\bm{Q}}^{(n)}_{k+m}
\right\Vert\\
&\le 
\left\Vert
\prod_{m=2}^{\ell} \bm{Q}^{(n)}_{k+m} - \prod_{m=2}^{\ell} \widehat{\bm{Q}}^{(n)}_{k+m}
\right\Vert
+
\left\Vert
\widehat{\bm{Q}}^{(n)}_{k+1}-\bm{Q}^{(n)}_{k+1} 
\right\Vert\\
&\le \cdots \le \ell \max_{m=1,\dots,\ell}\left\Vert
\widehat{\bm{Q}}^{(n)}_{k+m}-\bm{Q}^{(n)}_{k+m}
\right\Vert
\end{align*}
By the Lipschitz condition \ref{lip_cond}, on the event $B^{(n)}$ we obtain that for all $k,m\ge 0$ with $k+m\le \lfloor n^{1+\varepsilon}\rfloor$,
\begin{align*}
\left\Vert
\widehat{\bm{Q}}^{(n)}_{k+m}-\bm{Q}^{(n)}_{k+m} 
\right\Vert \le\frac{K_L}{n}\left|\chi^{(n)}_k-\frac{k}{n}\right| \le K_L \kappa(q,\varepsilon)(\log n)n^{-3/2+\varepsilon}.
\end{align*}
Then,
\begin{align*}
& \sum_{\substack{k,\ell \ge 0\\ k+\ell \le \lfloor n^{1+\varepsilon}\rfloor }} \mathrm{Poi}_{ns}(k)\times \mathrm{Poi}_{n(t-s)}(\ell)\times\left\Vert\prod_{m=1}^{\ell} \bm{Q}^{(n)}_{k+m} - \prod_{m=1}^{\ell} \widehat{\bm{Q}}^{(n)}_{k+m} \right\Vert\\
&\quad\le \sum_{k,\ell \ge 0} \mathrm{Poi}_{ns}(k)\times \mathrm{Poi}_{n(t-s)}(\ell)\times \ell K_L \kappa(q,\varepsilon)(\log n)n^{-3/2+\varepsilon}\\
&\quad= K_L \kappa(q,\varepsilon)(\log n)n^{-3/2+\varepsilon} \sum_{\ell\ge0}  \ell\,\mathrm{Poi}_{n(t-s)}(\ell) = (t-s) K_L \kappa(q,\varepsilon)(\log n)n^{-1/2+\varepsilon}.
\end{align*}
Note that 
\begin{align*}
\sum_{\substack{k,\ell \ge 0, k+\ell > \lfloor n^{1+\varepsilon}\rfloor }} \mathrm{Poi}_{ns}(k)\times \mathrm{Poi}_{n(t-s)}(\ell)= \sum_{\ell > \lfloor n^{1+\varepsilon}\rfloor } \mathrm{Poi}_{nt}(\ell) = \P (\chi^{(n)}_{\lfloor n^{1+\varepsilon}\rfloor} \le t )\le \P (\chi^{(n)}_{\lfloor n^{1+\varepsilon}\rfloor} \le T ),
\end{align*} where the last expression can be verified to be an $o(n^{-1})$ function by virtue of (\ref{eq:rateaux2}). Using the triangle inequality and taking supremums, we get that on $B^{(n)}$
\begin{align*}
\sup_{0\le s<t\le T}\Vert \bm{P}^{(n)}(s,t)-\widehat{\bm{P}}^{(n)}(s,t)\Vert &\le T K_L \kappa(q,\varepsilon)(\log n)n^{-1/2+\varepsilon} + \P (\chi^{(n)}_{\lfloor n^{1+\varepsilon}\rfloor} \le T )\\
& \le (T+1) K_L \kappa(q,\varepsilon)(\log n)n^{-1/2+\varepsilon},
\end{align*}
where the last equality holds for large enough $n$ such that $\tfrac{\P (\chi^{(n)}_{\lfloor n^{1+\varepsilon}\rfloor} \le T )}{  T K_L \kappa(q,\varepsilon)(\log n)n^{-1/2+\varepsilon}}\le 1$. We thus conclude that \eqref{eq:ratePhat} holds with $\kappa_L(q,\varepsilon):=2 K_L \kappa(q,\varepsilon)$. Equation \eqref{eq:ratePtilde} follows in a similar fashion by noting that on $B^{(n)}$ and for all $k,m\ge 0$ with $k+m\le \lfloor n^{1+\varepsilon}\rfloor$,
\begin{align*}
\left\Vert
\widetilde{\bm{Q}}^{(n)}_{k+m}-\bm{Q}^{(n)}_{k+m} 
\right\Vert & \le \frac{K_L}{n} \left(\left| \chi^{(n)}_{k+m}- \frac{k+m}{n} \right| + \E\left(\left| \chi^{(n)}_{k+m}- \frac{k+m}{n} \right|\right)\right)\\
& \le \frac{K_L}{n} \left(\left| \chi^{(n)}_{k+m}- \frac{k+m}{n} \right| + \sqrt{\E\left(\left| \chi^{(n)}_{k+m}- \frac{k+m}{n} \right|^2\right)}\right)\\
& \le \frac{K_L}{n} \left(\kappa(q,\varepsilon)(\log n)n^{-1/2+\varepsilon} + \frac{\sqrt{k+m}}{n}\right)\le 2 K_L \kappa(q,\varepsilon)(\log n)n^{-3/2+\varepsilon};
\end{align*}
{\red please note that in the third inequality we employed the standard deviation formula for $\chi^{(n)}_{k+m}$.}
$\hfill\square$\end{proof}
}
{\red

\begin{remark}\label{rem:ua1}
Formulas \eqref{eq:Phat7} and \eqref{eq:Ptilde7} are intrinsically connected to the uniformization method which is also used to analyze the uniform acceleration of time-inhomogeneous Markov jump processes in \cite{massey1998uniform}. Specifically, \eqref{eq:Phat7} and \eqref{eq:Ptilde7} can be compared with  \cite[Eq. (3.11)]{massey1998uniform}, although two significant aspects distinguish them from our framework. First, the authors in the mentioned work focus on studying a time-inhomogeneous Markov jump process that is inspected at a dense set of Poisson observations, while its associated generator is proportionally scaled in a simultaneous manner. This construction creates a process that is distinct from the original one, in fact, an accelerated version of it. The distinction is confirmed when noting that their uniform acceleration approximations yield a discrete Markov chain with transition probabilities that remain constant with respect to the scaling factor (cf. Corollary 3.2 in \cite{massey1998uniform}). In contrast, within our framework, the associated discrete Markov chains possess transition probabilities $\widehat{\bm{Q}}^{(n)}_\ell$ and $\widetilde{\bm{Q}}^{(n)}_\ell$ that depend on $n$, a consequence of keeping the infinitesimal generator fixed as opposed to accelerating it. Second, in \cite[Eq. (3.11)]{massey1998uniform}, the authors address multiple integrals of transition probabilities. Our method sidesteps this by considering random matrices ${\bm{Q}}^{(n)}_\ell$ in the conditional case or the matrices $\widehat{\bm{Q}}^{(n)}_\ell$ and $\widetilde{\bm{Q}}^{(n)}_\ell$ in the unconditional case.

We believe that accelerating our generator matrices could provide an interesting and alternative construction which could form the basis for new strong approximations of inhomogeneous Markov jump-processes; this is delegated to future research.
\end{remark}
}

\section{Applications to inhomogeneous phase-type distributions}\label{sec:IPH}

Let $\mathcal{J}$ be a transient time-inhomogeneous Markov jump process defined on the state-space $\{1,\dots, p\}$ driven by an subintensity matrix function $\bm{S}(t)$. As pointed out in Remark \ref{rem:star1}, $\mathcal{J}$ can be considered instead as a non-terminating time-inhomogeneous Markov jump process on the state-space $\{1,\dots,p\}\cup\{\star\}$ driven by the $(p+1)\times (p+1)$ intensity matrix
\begin{align*}
	\left( \begin{array}{cc}
		\bfS(t) &  \bfs(t) \\
		\0 & 0
	\end{array} \right),\quad t\ge 0,
\end{align*}
where $\bfs(t)=- \bfS(t) \, \bfe$ is a $p$--dimensional column vector providing the exit rates from each state directly to $\star$ (here $\bfe $ is a $p$--dimensional column vector of ones). Note that in this framework, termination of $\mathcal{J}$ is equivalent to its absoption to $\star$; without loss of generality, we will use these two notions interchangeably.


\begin{definition}\label{def_iph}
Define the absorption time of the time-inhomogeneous Markov jump process as
\begin{align*}
	\tau = \inf \{ t >  0 : J(t) = \star \}.
\end{align*}
Then we say that $\tau$ follows an \emph{inhomogeneous phase--type distribution} with representation $(\vect{\alpha},(\bfS(t))_{t\ge0})$, or simply $\tau \sim \mbox{IPH}(\vect{\alpha},\bfS(t))$.
\end{definition}
The class of inhomogeneous phase-type (IPH) distributions was defined in \cite{albrecher2019inhomogeneous} as a generalization of the widely used class of phase-type (PH) distributions \cite{neuts1975probability}, for which the underlying Markov jump process $\mathcal{J}$ is allowed to be time-homogeneous only. Both PH and IPH distributions are dense on the set of all distributions with positive support \cite[Section 3.2.1]{bladt2017matrix}. However, the IPH class can target any tail behaviour, while PH distributions are only suitable for random phenomena whose tail distribution decays at some exponential rate. This generality of IPH over PH distributions comes at a cost: while the PH class has been shown to be the archetype distribution for which many descriptors are available in explicit matrix-form (see e.g. \cite{bladt2017matrix}), the  formulae currently available for IPH distributions are considerably more scarce. While this scarcity is expected to be less severe as the systematic study of IPH distributions in the literature matures, the complexity associated to product integrals (when compared with the matrix-exponential function) will likely not allow for the IPH theory to be as rich as the PH one. Here we aim to amend this by employing Theorem \ref{th:strongmain1} to provide approximations to any IPH distribution via \emph{infinite-dimensional phase-type distributions} (\cite{shi2005sph}), the latter being a generalization of the PH class which allows for the underlying state-space to have countably-infinite cardinality.

\begin{theorem}[Absorption time convergence]\label{abs_time_theo}
Let $\tau \sim \mbox{IPH}(\vect{\alpha},\bfS(t))$ where the entries of $\bfS(t)$ are uniformly bounded. Then there exists a sequence of random variables $\{\tau^{(n)}\}_{n\ge \lambda_0}$ such that $\tau^{(n)}\rightarrow \tau$ in probability as $n\rightarrow\infty$, where each $\tau^{(n)}$ {  conditionally on $\mathcal{N}^{(n)}$} follows an infinite-dimensional phase-type distribution with initial distribution $\vect{\alpha}^{(n)}=(\vect{\alpha},\bm{0},\bm{0},\dots)$ and subintensity matrix function
\begin{equation}\label{eq:iph_struct}
{\mat{S}}^{(n)}=\begin{pmatrix}-n\bm{I}&n\bm{Q}^{(n)}_1& &&\\
&-n\bm{I}&n\bm{Q}^{(n)}_2& &\\
&&-n\bm{I}&n\bm{Q}^{(n)}_3&\\
&&&\ddots&\ddots
\end{pmatrix},\end{equation}
where
\begin{equation}\label{eq:Qnell1_IPH}
{  \bm{Q}_\ell^{(n)}= \bm{I}+\frac 1n  \bm{S}(\chi^{(n)}_\ell)}.\end{equation}

Furthermore, the rate of convergence of $\tau^{(n)}$ to $\tau$ is given by
\begin{equation}\label{eq:rate_convergence_PH}\mathbb{P}\left(|\tau-\tau^{(n)}|>\kappa(q,\varepsilon)(\log n)n^{-1/2+{  \varepsilon/2}}\right)= o(n^{-q}) + \mathbb{P}(\tau> (1-\varepsilon)n^{\varepsilon})\quad\mbox{for all}\quad\varepsilon>0\end{equation}
as $n\to \infty$.
\end{theorem}
\begin{proof}\textbf{Proof.}
Let us borrow the probability space $(\Omega,\mathcal{F},\mathbb{P})$ of Section \ref{sec:strongMJP}, along with the random variables and stochastic processes defined there. For $\tau=\inf\{ t>0 : J(t)=\star\}$ and $\tau^{(n)}=\inf\{ t>0 : J^{(n)}(t)=\star\}$, we will prove that (\ref{eq:rate_convergence_PH}) holds.

Fix $n\ge \lambda_0$. Recall that by construction, all the jumps of $\mathcal{J}$ can only occur at the random time grid $\{\chi^{(n)}_\ell\}_{\ell\ge 0}$, meaning that there exists a (random) unique $\gamma_n\ge 1$ such that $\tau=\chi^{(n)}_{\gamma_n}$. {  Similarly, since the jumps of $\mathcal{J}^{(n)}$ can only occur at the random time grid $\{\theta^{(n)}_\ell\}_{\ell\ge 0}$, and {\red $J^{(n)}(\theta^{(n)}_\ell)=J( \chi^{(n)}_\ell)$} for all $\ell\ge 0$, then $\tau^{(n)}=\theta^{(n)}_{\gamma_n}$}. Thus,
\begin{align}
\mathbb{P}&\left(|\tau-\tau^{(n)}|>\kappa(q,\varepsilon)(\log n)n^{-1/2+{  \varepsilon/2}}\right)\nonumber\\
&  = \mathbb{P}\left(|\chi^{(n)}_{\gamma_n}-\theta^{(n)}_{\gamma_n}|>\kappa(q,\varepsilon)(\log n)n^{-1/2+{  \varepsilon/2}}\right)\nonumber\\
& =\mathbb{P}\left(|\chi^{(n)}_{\gamma_n}-\theta^{(n)}_{\gamma_n}|>\kappa(q,\varepsilon)(\log n)n^{-1/2+{  \varepsilon/2}}, \gamma_n\le \lfloor n^{1+\varepsilon}\rfloor\right)\nonumber\\
&\quad + \mathbb{P}\left(|\chi^{(n)}_{\gamma_n}-\theta^{(n)}_{\gamma_n}|>\kappa(q,\varepsilon)(\log n)n^{-1/2+{  \varepsilon/2}}, \gamma_n> \lfloor n^{1+\varepsilon}\rfloor\right)\nonumber\\
&\le \mathbb{P}\left(\max_{\ell\in\{1,\dots, \lfloor n^{1+\varepsilon}\rfloor\}}\left|\chi_\ell^{(n)}-\theta_\ell^{(n)}\right|\ge \kappa(q,\varepsilon)(\log n)n^{-1/2+{  \varepsilon/2}}\right) + \mathbb{P}\left(\gamma_n>\lfloor n^{1+\varepsilon}\rfloor\right).\label{eq:rate_PH_aux1}
\end{align}
By Lemma \ref{lem:ratestrong1}, the first summand in the r.h.s. of (\ref{eq:rate_PH_aux1}) is an $o(n^{-q})$ function. Meanwhile, 
\begin{align}
\mathbb{P}\left(\gamma_n>\lfloor n^{1+\varepsilon}\rfloor\right)& = \mathbb{P}\left(\chi^{(n)}_{\gamma_n}>\chi^{(n)}_{\lfloor n^{1+\varepsilon}\rfloor}\right) = \mathbb{P}\left(\tau>\chi^{(n)}_{\lfloor n^{1+\varepsilon}\rfloor}\right)\nonumber\\
& \le \mathbb{P}\left(\tau>\tfrac{\lfloor n^{1+\varepsilon}\rfloor}{n} - \left|\chi^{(n)}_{\lfloor n^{1+\varepsilon}\rfloor}-\tfrac{\lfloor n^{1+\varepsilon}\rfloor}{n}\right|\right)\nonumber\\
& \le \mathbb{P}\left(\tau>\tfrac{\lfloor n^{1+\varepsilon}\rfloor}{n} - \left|\chi^{(n)}_{\lfloor n^{1+\varepsilon}\rfloor}-\tfrac{\lfloor n^{1+\varepsilon}\rfloor}{n}\right|, \left|\chi^{(n)}_{\lfloor n^{1+\varepsilon}\rfloor}-\tfrac{\lfloor n^{1+\varepsilon}\rfloor}{n}\right| < \kappa(q,\varepsilon)(\log n)n^{-1/2+{  \varepsilon/2}} \right)\nonumber\\
&\quad + \mathbb{P}\left(\tau>\tfrac{\lfloor n^{1+\varepsilon}\rfloor}{n} - \left|\chi^{(n)}_{\lfloor n^{1+\varepsilon}\rfloor}-\tfrac{\lfloor n^{1+\varepsilon}\rfloor}{n}\right|, \left|\chi^{(n)}_{\lfloor n^{1+\varepsilon}\rfloor}-\tfrac{\lfloor n^{1+\varepsilon}\rfloor}{n}\right| \ge \kappa(q,\varepsilon)(\log n)n^{-1/2+{  \varepsilon/2}}\right)\nonumber\\
& \le \mathbb{P}\left(\tau>\tfrac{\lfloor n^{1+\varepsilon}\rfloor}{n} - \kappa(q,\varepsilon)(\log n)n^{-1/2+{  \varepsilon/2}} \right)\\
&\quad+ \mathbb{P}\left(\left|\chi^{(n)}_{\lfloor n^{1+\varepsilon}\rfloor}-\tfrac{\lfloor n^{1+\varepsilon}\rfloor}{n}\right| \ge \kappa(q,\varepsilon)(\log n)n^{-1/2+{  \varepsilon/2}}\right).\label{eq:rate_PH_aux2}
\end{align}
Once again, by Lemma \ref{lem:ratestrong1} the second summand in the r.h.s. of (\ref{eq:rate_PH_aux2}) is an $o(n^{-q})$ function, while the first summand is smaller or equal than $\mathbb{P}\left(\tau>(1-\varepsilon)n^\varepsilon \right)$ for sufficiently large $n$. Thus, (\ref{eq:rate_convergence_PH}) holds and so does the convergence in probability of $\tau^{(n)}$ to $\tau$.
$\hfill\square$\end{proof}

\begin{remark} As {  opposed} to the result presented in Theorem \ref{th:strongmain1}, the convergence obtained in Theorem \ref{abs_time_theo} is not strong, but in a probability sense. However, if $\mathbb{P}(\tau > (1-\varepsilon)n^\varepsilon)$ decreases to $0$ fast enough (say, at a rate proportional to $n^{-1-\varepsilon}$) then the convergence can be upgraded to strong via standard Borel-Cantelli arguments.
\end{remark}

{  Given that the diagonal elements of (\ref{eq:iph_struct}) are given by $-n$, and that the off-diagonal elements of each row are nonnegative and sum at most $n$, then ${\mat{S}}^{(n)}$ can be regarded as a bounded linear operator w.r.t. the $\infty$-norm for matrices, and thus, $\exp({\mat{S}}^{(n)}t) := \sum_{\ell=0}^\infty \tfrac{({\mat{S}}^{(n)}t)^\ell}{\ell !}$ is well defined and finite for each $t\ge 0$}. By simple arguments analogous to those employed in \cite{shi2005sph} and {  \cite[Eq. (2.5)]{bladt2015calculation}}, {   the conditional density function $f^{(n)}$ associated to $\tau^{(n)}$ given $\mathcal{N}^{(n)}$} then takes the form
\[f^{(n)}(t) = \vect{\alpha}^{(n)}\exp\left(\bfS^{(n)}t\right)\bfs^{(n)},\quad t\ge 0,\]
where $\bfs^{(n)}=-\bfS^{(n)}\bfe_\infty$ and $\bfe_\infty$ denotes an infinite-dimensional column vector of ones. Below we provide an equivalent form for $f^{(n)}$ which bypasses the need to manipulate infinite-dimensional matrix-exponentiation and multiplication. 

\begin{proposition}[Density of the approximation]\label{dens_prop}
Let $\tau^{(n)}$ be the approximation of Theorem \ref{abs_time_theo}. Then
\begin{align}\label{eq:ftaun}
f^{(n)}(t)&=\sum_{\ell=1}^\infty \left[\vect{\alpha}\bm{Q}^{(n)}_1\cdots\bm{Q}^{(n)}_{\ell-1}(\mat{I}-\bm{Q}_\ell^{(n)})\bfe\right] \frac{t^{\ell-1}}{(\ell-1)!}
n^\ell\exp(-n t) ,\quad t\ge0.
\end{align}
\end{proposition}
\begin{proof}\textbf{Proof.}
{  The main argument may be borrowed from Theorem \ref{th:transition_prob1}, however we instead give an analytic proof which also provides insight into our construction.} 

Truncating the infinite-dimensional vector and matrix $(\vect{\alpha}^{(n)},\bfS^{(n)})$ at the $m$-th block,
we obtain a defective finite-dimensional phase--type distribution with representation of the form
\[  \vect{\alpha}^{(n,m)}=(\vect{\alpha},\vect{0},\dots,\vect{0}), \ \ \mat{S}^{(n,m)}=
\begin{pmatrix}
-n\mat{I} & n \bm{Q}_1^{(n)} & 0 & ... & 0 \\
0 & -n\mat{I} & n \bm{Q}_2^{(n)}  & ... & 0 \\
0 & 0 & -n\mat{I} & ... & 0 \\
\vdots & \vdots & \vdots & \ddots & \vdots\\
0 & 0 & 0 & ... & -n\mat{I}
\end{pmatrix} ,
   \]
  and by block-multiplication, it is not hard to see that for all $z\in\mathbb{C}\setminus\{-n\}$,  \[  (z\mat{I}-\mat{S}^{(n,m)})^{-1} =
\begin{pmatrix}
\frac{\mat{I}}{z+n} & \frac{n \bm{Q}_1^{(n)}}{(z+n)^2} & \frac{\bm{Q}_1^{(n)}\bm{Q}_2^{(n)} n^2}{(z+n)^3} & \cdots & \frac{\bm{Q}_1^{(n)}\cdots \bm{Q}^{(n)}_{m-1}n^{m-1}}{(z+n)^m} \\
0 & \frac{\mat{I}}{z+n} & \frac{n \bm{Q}_2^{(n)}}{(z+n)^2} & \cdots & \frac{\bm{Q}_2^{(n)}\cdots \bm{Q}_{m-1}^{(n)}n^{m-2}}{(z+n)^{m-1}} \\
0 & 0 & \frac{\bf{I}}{z+n} & \cdots & \frac{\bm{Q}_3^{(n)}\cdots\bm{Q}_{m-1}^{(n)} n^{m-3}}{(z+n)^{m-1}} \\
\vdots & \vdots & \vdots & \ddots & \vdots\\
0 & 0 & 0 & \cdots & \frac{\mat{I}}{z+n}
\end{pmatrix}.
     \]
An analogous expression holds for the infinite-dimensional inverse of $(z\mat{I}-\mat{S}^{(n)})$, so that we may use holomorphic functional calculus for bounded operators (see e.g. \cite{haase2006functional}) and the residue theorem to get
 \begin{eqnarray*}
\vect{\alpha}^{(n)} \exp({\mat{S}}^{(n)}t)\vect{s}^{(n)}&=&\frac{1}{2\pi \ii}\int_\Gamma \exp(z)\vect{\alpha}^{(n)}(z\mat{I}-t \mat{S}^{(n)})^{-1}\vect{s}^{(n)}dz  \\
&=&  \frac{1}{2\pi \ii}\int_\Gamma \exp(z) \vect{\alpha}\left[ \sum_{\ell=1}^\infty
\frac{(n\bm{Q}_1^{(n)} t)(n\bm{Q}_2^{(n)} t)\cdots (n\bm{Q}_{\ell-1}^{(n)}t)n(\mat{I}-\bm{Q}_\ell^{(n)})\bfe}{(z+n t)^{\ell}}
 \right]dz \\ 
&=& \sum_{\ell=1}^\infty n^\ell\left[\vect{\alpha}\bm{Q}^{(n)}_1\cdots\bm{Q}^{(n)}_{\ell-1}(\mat{I}-\bm{Q}_\ell^{(n)})\bfe\right]t^{\ell-1}
\frac{1}{2\pi \ii}\int_\Gamma  \frac{\exp(z)}{(z+n t)^{\ell}}dz \\
&=& \sum_{\ell=1}^\infty \left[\vect{\alpha}\bm{Q}^{(n)}_1\cdots\bm{Q}^{(n)}_{\ell-1}(\mat{I}-\bm{Q}_\ell^{(n)})\bfe\right]\frac{t^{{  \ell}-1}}{(\ell-1)!}
n^\ell\exp(-n t), \\
\end{eqnarray*}  
where $\Gamma$ is a path enclosing $\{-tn\}$, the (only) eigenvalue of $t\mat{S}^{(n)}$.
$\hfill\square$\end{proof}

\begin{remark} 
If we impose Lipschitz continuity on the subintensity matrix, we may replace {\red $\bm{Q}^{(n)}_\ell$} in Proposition \ref{dens_prop} by either of:
\begin{equation}\label{eq:Qnell2_IPH}
\widehat{\bm{Q}}^{(n)}_\ell=\bm{I}+\frac 1n \left[ \bm{S}(\ell/n) \right],\quad \widetilde{\bm{Q}}^{(n)}_\ell=\bm{I}+\frac 1n \E\left[ \bm{S}(\chi^{(n)}_\ell) \right],
\end{equation}
to obtain weakly convergent unconditional distributions. The associated unconditional densities will be denoted by $\hat f^{(n)}$ and $\tilde f^{(n)}$, respectively.

For this reason, in the remainder of the applications, we will purposefully refrain from explicitly stating whether the approximation is conditional or unconditional, and when dealing with the unconditional case, we specialize only on $\tilde f^{(n)}$. Analysis using $\hat f^{(n)}$ yields qualitatively similar results and thus are omitted. {\red Also note that these densities converge to $f^{(n)}$ due to the fact that the transition probabilities converge pointwise for all $t\ge 0$ (Theorem \ref{gen_convs}), and thus, the corresponding absorption times converge weakly.
}
\end{remark}

Since convergence in probability implies convergence in distribution, Theorem \ref{abs_time_theo} together with Proposition \ref{dens_prop} provide a method to approximate an IPH distribution with bounded subintensity matrix function via a sequence of infinite-dimensional phase-type distributions with density given by (\ref{eq:ftaun}). When the matrices $\bfS(s)$ and $\bfS(t)$ commute for any $s,t$ we may further simplify the transition matrix of $\mathcal{J}$ in $\{1,\dots,p\}\cup\{\star\}$ from time $s$ to $t$, $s<t$, into the form 
\begin{align*}
\begin{pmatrix}
    \exp\left(\int_s^t\bfS(u)\dd u\right) & \bfe-\exp\left(\int_s^t\bfS(u)\dd u\right) \bfe \\
    \textbf{0} & 1
    \end{pmatrix};
\end{align*}
see \cite{albrecher2019inhomogeneous}. A narrower class that allows for effective statistical analysis arises from imposing the following simplification:
$$\bfS(t) = \lambda(t)\,\bfS,\quad t\ge 0,$$ where $\lambda:\mathds{R}_+\rightarrow\mathds{R}_+$ is some locally integrable inhomogeneity function, and $\bfS$ is a fixed subintensity matrix that does not depend on time $t$. For this reduced case, the standard notation is then \begin{align}
\tau \sim  \mbox{IPH}(\vect{\alpha} , \bfS , \lambda ){  .}
\end{align} Note that PH distributions are a special case when $\lambda\equiv 1$. Below we explore two examples of $\mbox{IPH}(\vect{\alpha} , \bfS , \lambda )$ distributions found in \cite{albrecher2019inhomogeneous} and provide their infinite-dimensional phase-type approximations.

\begin{example}[Approximating the Matrix Gompertz distribution]\rm
Consider the case where we have $\tau\sim \mbox{IPH}(\vect{\alpha} , \bfS , \lambda )$ with $\lambda(s)=\exp(\beta s),$ $\beta>0$. This distribution has an asymptotically Gompertz hazard rate. Then, for $n>\beta$ we have
$$\int_0^\infty \frac{ (n s)^{\ell-1}}{(\ell -1)!}e^{-ns}\lambda(s) \, \dd s =\frac{1}{n}\left(\frac{n}{n-\beta}\right)^\ell,$$
from which
\begin{equation}\label{eq:iph_struct_gomp}
{\widetilde{\mat{S}}}^{(n)}
=\begin{pmatrix}-n\bm{I}& \left(\frac{n}{n-\beta}\right)\bfS+n\bm{I}& &&\\
&-n\bm{I}&\left(\frac{n}{n-\beta}\right)^2\bfS+n\bm{I}& &\\
&&-n\bm{I}&\left(\frac{n}{n-\beta}\right)^3\bfS+n\bm{I}&\\
&&&\ddots&\ddots
\end{pmatrix}.\end{equation}
In the left panel of Figure \ref{plots:iph_approxs} such an approximation is carried out for varying truncation limits\footnote{By truncation limit we refer to a finite-dimensional approximation to an infinite-dimensional matrix, where the chosen finite order \textit{truncates} the number of infinite blocks.} and with $n=20$. The original IPH parameters are given by
\begin{align*}
	\vect{\alpha}=(0.42,\, 0.58),\quad \mat{S}(t)= \exp(t)\left( \begin{array}{cc}
		-0.78 & 0.57 \\
		0.91 & -1.81
	\end{array} \right),\quad t\ge0.
\end{align*}

\end{example}

\begin{example}[Approximating a Matrix Weibull distribution]\rm \label{ex:2_weibull}
Now consider $\tau\sim \mbox{IPH}(\vect{\alpha} , \bfS , \lambda )$ with $\lambda(s)=\beta s^{\beta-1},$ $\beta>0$. This distribution has asymtotic Weibull hazard rate. Then
$$\int_0^\infty \frac{ (n s)^{\ell-1}}{(\ell -1)!}e^{-ns}\lambda(s) \, \dd s =\frac{\beta {   n^{-\beta}}\Gamma(\ell+\beta-1)}{(\ell-1)!},$$
from which
\begin{equation}\label{eq:iph_struct_weibull}
{\widetilde{\mat{S}}}^{(n)}
=\begin{pmatrix}-n\bm{I}& \frac{\beta n^{1-\beta}\Gamma(\beta)}{0!}\bfS+n\bm{I}& &&\\
&-n\bm{I}&\frac{\beta n^{1-\beta}\Gamma(\beta+1)}{1!}\bfS+n\bm{I}& &\\
&&-n\bm{I}&\frac{\beta n^{1-\beta}\Gamma(\beta+2)}{2!}\bfS+n\bm{I}&\\
&&&\ddots&\ddots
\end{pmatrix}.\end{equation}
In the right panel of Figure \ref{plots:iph_approxs} we illustrate such an approximation for varying truncation limits and with $n=100$. The original IPH parameters are given by
\begin{align*}
	\vect{\alpha}=(0.5,\, 0.5),\quad \mat{S}(t)= 3t^2 \left( \begin{array}{cc}
		-3 & 0.1 \\
		0.01 & -0.1
	\end{array} \right),\quad t\ge0.
\end{align*}

\end{example}

\begin{figure}[!htbp]\centering
\includegraphics[width=0.49\textwidth]{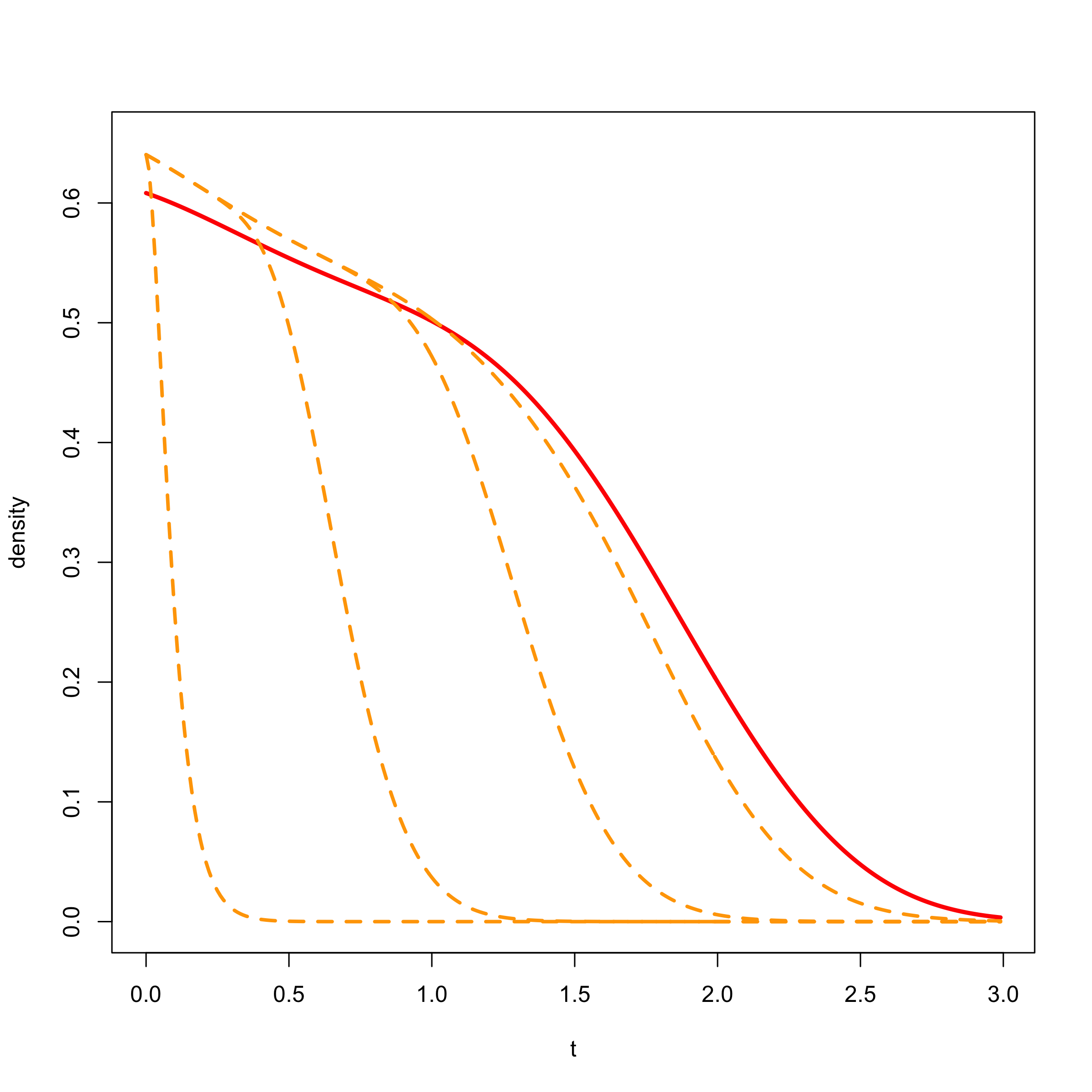}
\includegraphics[width=0.49\textwidth]{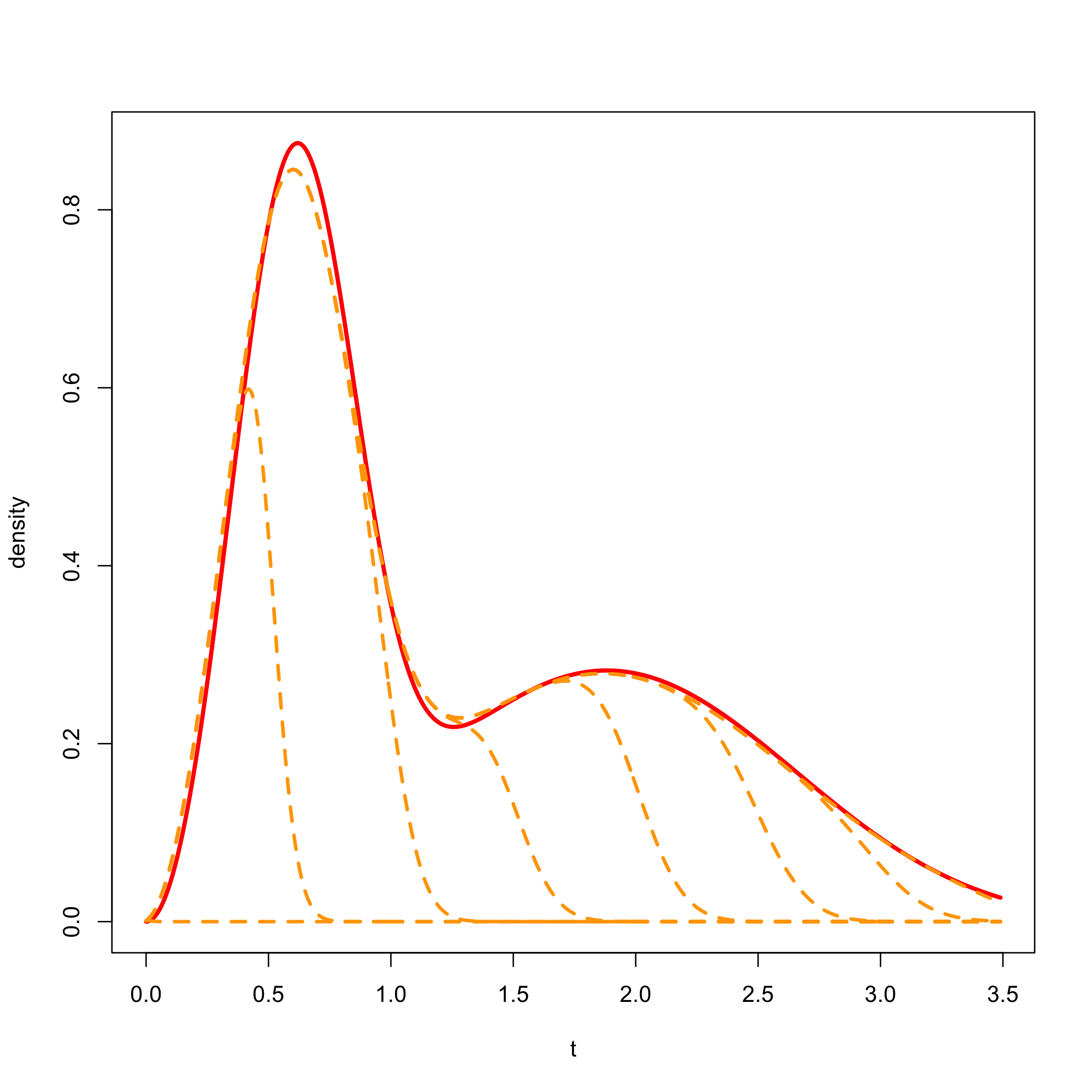}
\caption{Density approximations to given IPH distributions. Left panel: approximation to Matrix-Gompertz with upper truncation limit of $2,15,27,40$ and $n=20$. Right panel: approximation to Matrix-Weibull with upper truncation limit of $2,  52, 102, 152, 202, 252, 302, 352$ and $n=100$.} \label{plots:iph_approxs}
\end{figure}

\begin{remark}\rm
Note that the above distributions do not {   a priori have a bounded hazard rate nor are Lipschitz. However, truncation ammends this issue}, by replacing the hazard rate $\lambda(s)$ with $\lambda(s) \wedge K$ where $K=K(n)\to\infty$ as $n\to \infty$. In that case, we need to compute $\int_0^\infty \frac{ (n s)^{\ell-1}}{(\ell -1)!}e^{-ns}(\lambda(s) \wedge K)\, \dd s$ in place of $\int_0^\infty \frac{ (n s)^{\ell-1}}{(\ell -1)!}e^{-ns}\lambda(s) \, \dd s$. Numerically, however, one can choose $K$ large enough such that the two integrals are  indistinguishable; here we implemented the latter, since it has closed form expressions.
\end{remark}

We now proceed to illustrate the above convergence from a different perspective, which comes in the form of ruin probabilities, as shown below.

\subsection{Approximating ruin probabilities in the Cram\'er-Lundberg model.}

Let us consider a process $\{R(t)\}_{t\ge 0}$ of the form
\begin{align}
R(t)=u+\rho t-\sum_{k=0}^{C(t)}\tau_k,\quad t\ge0,
\end{align}
where $u\ge 0$, $\rho>0$, $\tau_k$ are i.i.d. positive random variables, and $\{C(t)\}_{t\ge 0}$ is an independent Poisson process with intensity $\nu$. Such a process is known in the risk theory literature as \emph{Cram\'er-Lundberg} (\cite{lundberg1903approximerad,cramer1955collective}), which is a simple model to describe how the initial capital $u$ of an insurance company increases in a piecewise fashion with a constant premium rate $\rho$ and decreases through jumps whose size matches the severity of claims $\tau_1,\tau_2,\dots$. 
A classic problem in the field consists in computing the \emph{infinite-horizon probability of ruin} as $u$ varies, defined by
\begin{align}
\psi(u):=\mathbb{P}\left(\inf_{t\ge0}\{R(t)\}<0\;|\;R({  0})=u\right),\quad u>0.
\end{align}
Regardless of its apparent simplicity, computing $\psi(u)$ is a challenging problem, with closed-form solutions available only for a handful of claim distributions, including the phase-type class (see e.g. \cite[Chapter IX]{asmussen2010ruin}). To the best of the authors' knowledge, the case when {  $\tau_k$} follows an IPH distribution has not been analyzed in the literature.

{  Here we propose to approximate each $\tau_k$ by a sequence of infinite-dimensional phase-type random variables $\{\tau^{(n)}_k\}_{n\ge1}$ obtained from {  the unconditional distributions of Section \ref{sec:uncond1}}, say with parameters $\bfS^{(n)}$ and $\vect{\alpha}^{(n)}$. Thus, we are interested in computing the infinite-horizon probability of ruin of the approximated risk process
\begin{align}
R^{(n)}(t)=u+\rho t-\sum_{k=0}^{C(t)}\tau^{(n)}_k,\quad t\ge 0.
\end{align}
We can then consider their infinite-horizon ruin probabilities
\begin{align}
\psi^{(n)}(u):=\P\left(\inf_{t\ge0}\{R^{(n)}(t)\}<0\;|\;R^{(n)}(0)=u\right)\quad u\ge 0,
\end{align}
as a sequence that approximately describes the original ruin probability $\psi(u)$.} By \cite{bladt2015calculation}, we have that
\begin{align}
\psi^{(n)}(u)=\vect{\alpha}^{(n)}_{-}\exp\left((\bfS^{(n)}+\bfs^{(n)}\vect{\alpha}^{(n)}_{-})u\right)\bfe,
\end{align}
where 
$\vect{\alpha}^{(n)}_{-}=\frac{\nu}{\rho} \vect{\alpha}^{(n)}[-\bfS^{(n)}]^{-1}$. In Figure \ref{plots:ruin_approx} we plot the above approximation for varying initial capital $u$, together with a Monte Carlo simulation for reference; we employ the IPH parameters from Example \ref{ex:2_weibull}.

\begin{figure}[!htbp]\centering
\includegraphics[width=0.7\textwidth]{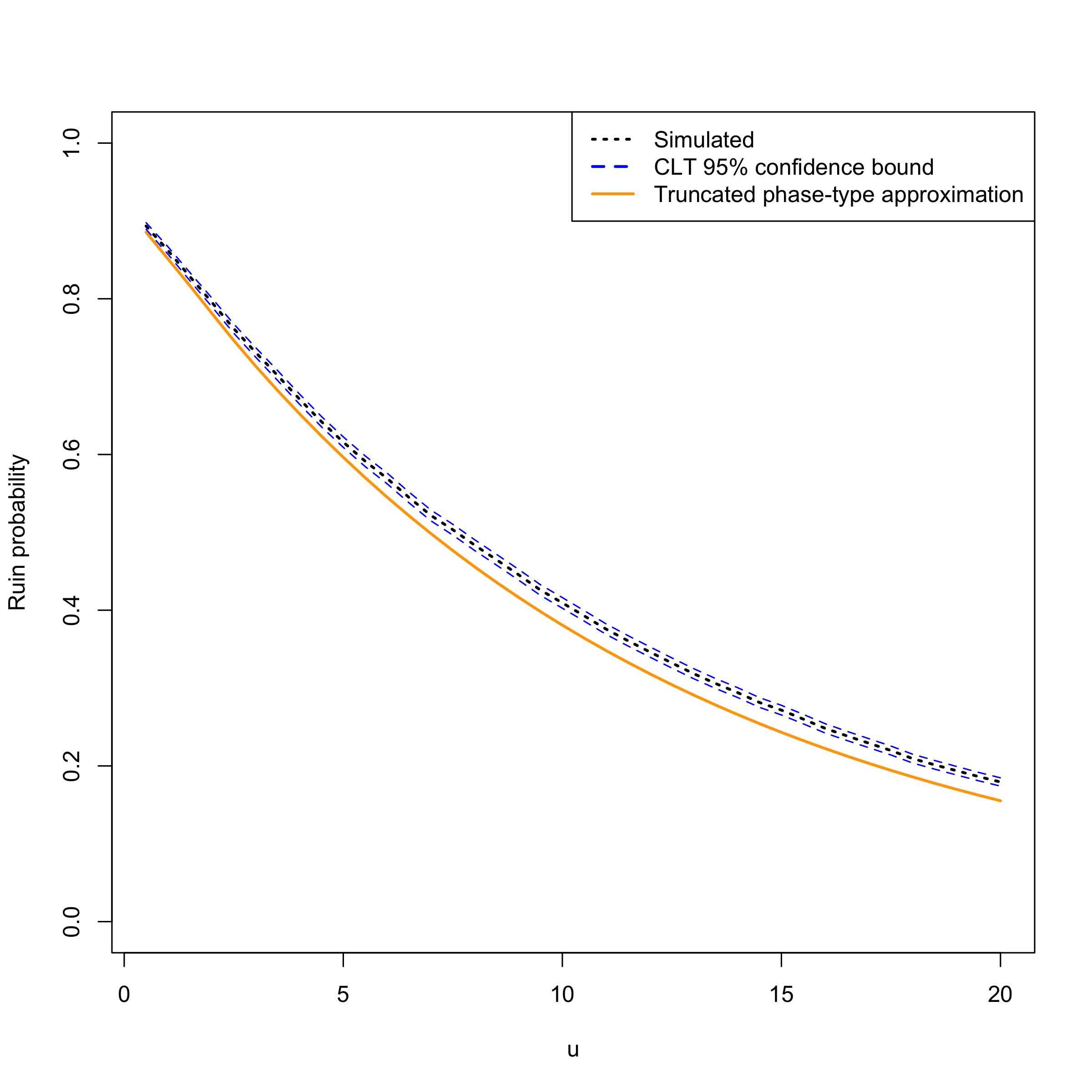}
\caption{Monte Carlo (from 20,000 simulations) versus approximated ruin probability for the Cram\'er-Lundberg model with IPH (Matrix-Weibull) distributed claim sizes, as a function of the initial capital $u$. The truncation limit is taken to be $352$ and $n=100$.}
\label{plots:ruin_approx}
\end{figure}

\subsection{Hazard rate approximations to arbitrary distributions and to data}

Any absolutely continuous distribution can be seen as a one-dimensional inhomogeneous phase-type distribution by considering a single-state state-space ($p=1$) and letting time run in accordance with the hazard rate of the target distribution. We may thus specialize the results from Theorem \ref{abs_time_theo} and Proposition \ref{dens_prop} to provide a novel (to the best of the authors' knowledge) infinite-dimensional phase-type approximation to arbitrary distributions of bounded hazard rate. We provide the details below.

\begin{corollary}[Hazard rate approximation]\label{abs_time_theo_uni}
Let $Y$ be an arbitrary random variable with a uniformly bounded hazard rate $h$, {  otherwise regarded as a one-dimensional $\mbox{IPH}$-distributed random variable of parameters $(\vect{\alpha} , \bfS , h )$ with $\vect{\alpha}=(1)$ and $\bfS=(-1)$}. Then there exists a sequence of random variables $\{Y^{(n)}\}_{n\ge \lambda_0}$ such that $Y^{(n)}\rightarrow Y$ in {  distribution} as $n\rightarrow\infty$, and where each $Y^{(n)}$ follows an infinite-dimensional Coxian phase-type distribution with initial distribution $\vect{\alpha}=(1,0,0,\dots)$ and subintensity matrix
\begin{equation}
{\widetilde{\mat{S}}}^{(n)}=\begin{pmatrix}-n&n\widetilde{Q}^{(n)}_1& &&\\
&-n&n\widetilde{Q}^{(n)}_2& &\\
&&-n&n\widetilde{Q}^{(n)}_3&\\
&&&\ddots&\ddots
\end{pmatrix},\end{equation}
where
\begin{equation}
\widetilde{Q}_\ell^{(n)}= 1-\int_0^\infty \frac{ (n s)^{\ell-1}}{(\ell -1)!}{  e^{-ns}}h(s) \, \dd s.\end{equation}
Furthermore, the density of the approximation is given by
\begin{align}
\tilde f_{Y^{(n)}}(y)&=
\sum_{\ell=1}^p\left((1-\widetilde{Q}^{(n)}_\ell)\prod_{m=1}^{\ell-1}\widetilde{Q}^{(n)}_m \right)\frac{y^{\ell-1}}{(\ell-1)!}
n^\ell\exp(-n y), \quad y\ge 0.
\end{align}
\end{corollary}

\begin{remark}\rm
When knowledge of $h$ is available, the integrals $\widetilde{Q}_\ell^{(n)}$ can be either explicitly or numerically computed. On the other hand, if $h$ is unknown and we are presented with data, we may use the following identity
\begin{align}\label{hrapprox}
\int_0^\infty \frac{ (n s)^{\ell-1}}{(\ell -1)!}{  e^{-ns}}h(s) \, \dd s\approx \int_0^\infty \frac{ (n s)^{\ell-1}}{(\ell -1)!}{  e^{-ns}} \, \dd \hat H(s),
\end{align}
where $\hat H$ is a suitable estimator of the cumulative hazard. For instance, in a fully observed setting, $\hat H$ can simply be the empirical cumulative hazard, whereas, in right-censored scenarios, we may use versions of the Nelson-Aalen estimator (see \cite{lawless2011statistical} for their definition and a general overview of survival analysis techniques). In Figure \ref{plots:real_data_approx} we illustrate this methodology with a multimodal simulated dataset.
\end{remark}

\begin{figure}[!htbp]\centering
\includegraphics[width=0.7\textwidth]{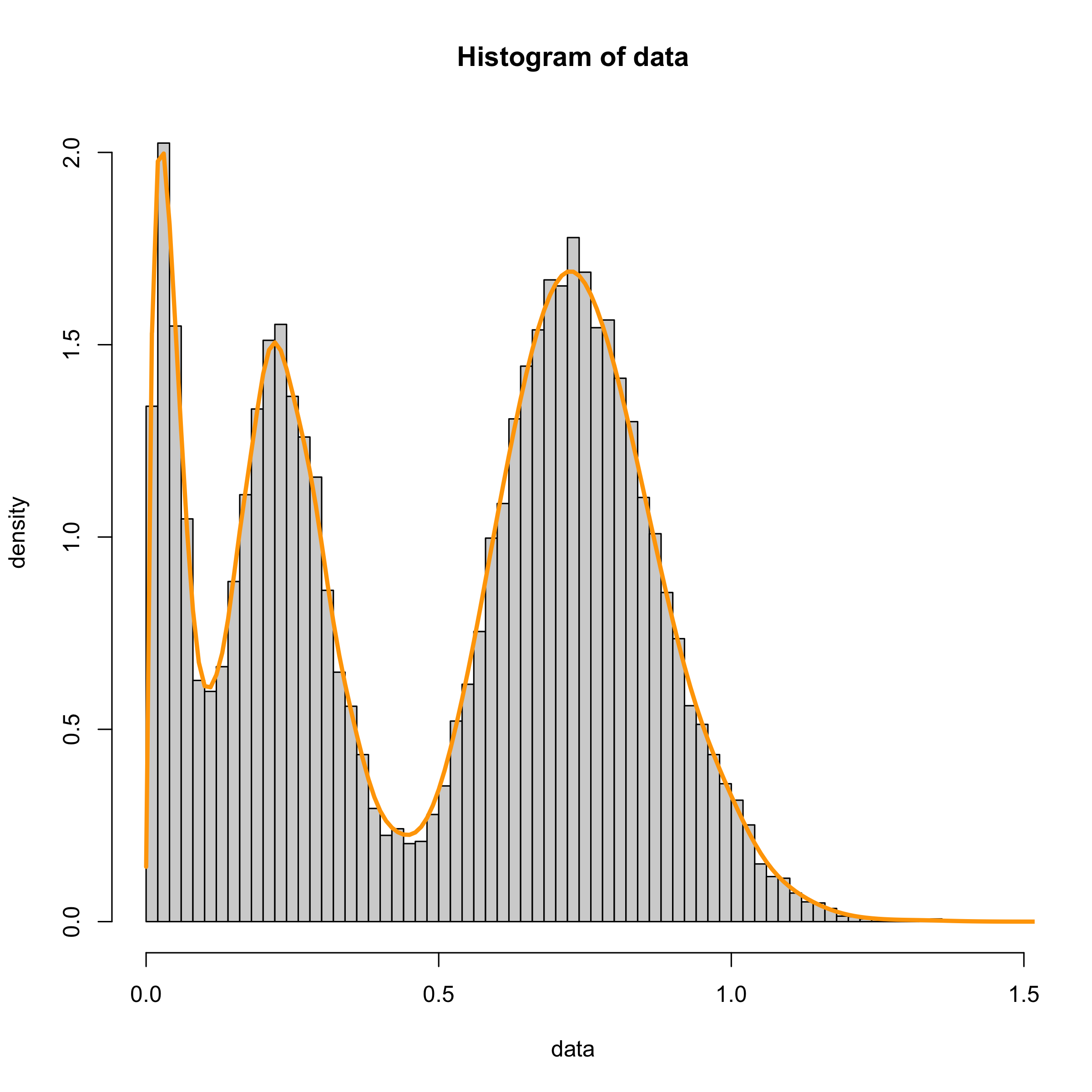}
\caption{Hazard rate approximation \eqref{hrapprox} to heterogeneous synthetic data. The simulation comprises $20,000$ $\Gamma(30,40)$ variables, $10,000$ 
$\Gamma(10,40)$ variables, and $5,000$ $\Gamma(2,40)$ variables. The approximation has $n=1500$ and truncation limit of $\lceil n M\rceil$, where $M$ is the largest data point in the sample. We have used $\hat H=-\log(1-\hat F)$, where $\hat F$ is the ecdf of the sample.}
\label{plots:real_data_approx}
\end{figure}

\subsection{Kulkarni's multivariate phase-type distributions}

Efforts to translate the success of PH distributions towards a multivariate setting have been done in \cite{assaf1984multivariate,kulkarni1989new,bladt2010multivariate,bladt2021tractable}, classes which enjoy different degrees of tractability and generality. Particularly, Kulkarni's multivariate phase-type distribution ($\mbox{MPH}^*$) constructed in \cite{kulkarni1989new} has enjoyed considerable attention due to its natural probabilistic interpretation, which we explain next. A nonnegative vector $(Y_1,\dots, Y_m)$ is said to be $\mbox{MPH}^*$-distributed if 
\begin{equation}\label{eq:defYMPHstar}Y_k\stackrel{d}{=}\int_0^\tau r(J(t),k) \, \dd t,\quad k\in\{1,\dots,m\},\end{equation}
where $\mathcal{J}=\{J(t)\}_{t\ge 0}$ is a terminating time-homogeneous Markov jump process on $\mathcal{E}$, $\tau$ is its termination time, and $\bm{R}=\{r(i,k)\}_{i\in\mathcal{E}, k\in\{1,\dots,m\}}$ is a nonnegative $p\times m$-dimensional matrix. The $\mbox{MPH}^*$ class enjoys a considerable level of flexibility, which together with the recently developed methods for data fitting in \cite{breuer2016semi} and \cite{albrecher2022fitting}, makes it an attractive option to consider for stochastic modelling {  purposes}. However, its main drawback is the lack of an explicit multivariate density function. Indeed, a multivariate density function can be only computed in certain subclasses, such as the one defined in \cite{assaf1984multivariate}, while the general case is (currently) only known as the solution to a system of partial differential equations (\cite{kulkarni1989new}).

Let us consider the multivariate random vector $(Y_1,\dots, Y_m)$ defined in (\ref{eq:defYMPHstar}) where $\mathcal{J}$ is instead a terminating \emph{time-inhomogeneous} Markov jump process of parameters with initial distribution $\bm{\pi}$ and subintensity matrix function $\bm{S}(t)$. Using the strongly convergent result in Theorem \ref{th:strongmain1}, here we present a way to {  conditionally} approximate the multivariate density function of $(Y_1,\dots, Y_m)$. While such a class differs from that defined in \cite{albrecher2022fitting} and has not been previously considered in the existing literature, it trivially generalizes the time-homogeneous version $\mbox{MPH}^*$. Thus, the methodology presented next provides a way to approximate density functions associated to multivariate phase-type distributions, which is novel even when reduced to the time-homogeneous case.
 
In order to provide the aforementioned multivariate density function approximation of $(Y_1,\dots, Y_m)$ with $\mathcal{J}$ a time-inhomogeneous Markov jump process, suppose that the probability space $(\Omega,\mathcal{F},\mathbb{P})$ introduced in Section \ref{sec:strongMJP} further supports the following independent components:
\begin{itemize}
  \item $m$ independent Poisson processes $\mathcal{M}_1\dots,\mathcal{M}_m$ with common intensity $\lambda_0$;
  \item $m$ independent sequences of independent Poisson processes $\{\widehat{\mathcal{M}}^{(\ell)}_k\}_{\ell=\lambda_0+1}^\infty$, $k\in\{1,\dots,m\}$, of parameter $1$.
\end{itemize}

For each $k\in\{1,\dots, m\}$ and $n\ge \lambda_0$, let $M_k^{(n)}$ denote the superposition of $\mathcal{M}_k$ and $\widehat{\mathcal{M}}_k^{(\lambda_0+1)},\dots, \widehat{\mathcal{M}}_k^{(n)}$, and let $\big\{\theta^{(n)}_{k,\ell}\big\}_{\ell\ge 0}$ correspond to the arrival times of $M_k^{(n)}$. For each $k\in\{1,\dots,m\}$, the Poisson process $M_k^{(n)}$ is independent of $\mathcal{N}^{(n)}$ and $\{U_\ell\}_{\ell =0}^\infty$, so that the strongly convergent result in Theorem \ref{th:strongmain1} still holds if we replace $\mathcal{J}^{(n)}$ with $\mathcal{J}^{(n)}_k$, where the latter is defined by
  \[J_k^{(n)}(t)= J\left(\chi^{(n)}_{\ell}\right)\quad\mbox{for}\quad t\in \left[\theta_{k,\ell}^{(n)}, \theta_{k,\ell+1}^{(n)}\right).\]
Then, the vector $(Y_1^{(n)},\dots, Y_m^{(n)})$ where
  \[Y_k^{(n)}=\int_0^{{  \tau^{(n)}_k}} r(J^{(n)}_k(t),k) \, \dd t,\quad k\in\{1,\dots,m\},\]
  (with {  $\tau^{(n)}_k$} equating the termination time of $\mathcal{J}^{(n)}_k$) converges strongly to 
  \begin{equation}\label{eq:timeinhomMPHstar1}\left(\int_0^{\tau} r(J(t),1) \, \dd t, \int_0^{\tau} r(J(t),2) \, \dd t,\dots, \int_0^{\tau} r(J(t),m) \, \dd t\right),\end{equation}
a multivariate phase-type random vector which allows for time-inhomogeneity. This strongly convergent result readily  implies that the multivariate density function of $(Y_1^{(n)},\dots, Y_m^{(n)})$ converges to that of (\ref{eq:timeinhomMPHstar1}), justifying its use as an approximation.  Furthermore, by construction, the processes $\mathcal{J}^{(n)}_1, \dots, \mathcal{J}^{(n)}_m$ are conditionally independent given {\red $\{J\left(\chi^{(n)}_{\ell}\right)\}_{\ell}$}, which allows us to explicitly express the {  conditional} {  multivariate} density function of $(Y_1^{(n)},\dots, Y_m^{(n)})$ as follows.

\begin{theorem}
  The random vector $(Y_1^{(n)},\dots, Y_m^{(n)})$ {  conditional on $\mathcal{N}^{(n)}$} follows the infinite mixture density
  \begin{equation}\label{eq:approxKulkarni1}f(x_1,x_2,\dots,x_m)= \sum_{\bm{i}=(i_1,\dots, i_p)\in\mathcal{Z}}\sum_{j\in\mathcal{E}} \alpha_{i_1,\dots,i_p; j} \beta_{|\bm{i}|, j}f_{1;i_1,\dots,i_p}(x_1)f_{2;i_1,\dots,i_p}(x_2)
  \cdots f_{m;i_1,\dots,i_p}(x_m),\end{equation}
  where  $|\bm{i}|= i_1+\cdots+ i_p$, $\mathcal{Z}=\{\bm{i}=(i_1,\dots,i_p): i_1,\dots, i_p\in\mathbb{N}_0, |\bm{i}|>0\}$, 
 {  
  \begin{align}
  \alpha_{i_1,\dots,i_p; j}& :=\mathbb{P}
  \left(\left\{J(\chi^{(n)}_{\ell})\right\}_{\ell= 0}^{|\bm{i}|-1}\mbox{ makes $i_a$ visits to $a$ for all $a\in\mathcal{E}$}, J(\chi^{(n)}_{|\bm{i}|-1})=j\mid \mathcal{N}^{(n)}\right),\\
  \beta_{\ell; j}& :=\mathbb{P}(J(\chi^{(n)}_{\ell})=\star \mid  J(\chi^{(n)}_{\ell-1})=j,\mathcal{N}^{(n)} )=1-\sum_{a\in\mathcal{E}} [\bm{Q}_\ell^{(n)}]_{j,a},\label{eq:approxKulkarnibeta1}
  \end{align}}
and $f_{k, i_1,\dots,i_p}$ denotes the hypoexponential distribution associated to $\sum_{j\in \mathcal{E}} \mathrm{Erl}(i_j, n/r(j,k))$, with each $\mathrm{Erl}(i_j, n/r(j,k))$ denoting an (independent) Erlang random variable of shape $i_j$ and rate $n/r(j,k)$. {   Here $\mathrm{Erl}(i, \infty)$ denotes the $0$-valued random variable for all $i\ge 1$}. Moreover, the constants $\alpha_{i_1,\dots,i_p; j}$ can be recursively computed by
  \begin{align}
  \alpha_{\bm{e}_h^\intercal; j}=\left\{\begin{array}{ccc}\pi_j&\mbox{if}& h=j\\ 0 &\mbox{if}& h\neq j\end{array}\right.\label{eq:approxKulkarnialpha0}
  \end{align}
  and for $|\bm{i}| \ge 2$,
  \begin{align}
 \alpha_{i_1,\dots,i_p; j}=\left\{\begin{array}{ccc}\sum_{h\in \mathcal{E}}\alpha_{i_1,\dots,i_j-1,
 \dots,i_p; h}[\bm{Q}_{|\bm{i}|-1}^{(n)}]_{hj}&\mbox{if}& i_j>0\\ 0 &\mbox{if}& i_j= 0.\end{array}\right.\label{eq:approxKulkarnialpharec1}
  \end{align} \end{theorem}
\begin{proof}\textbf{Proof.}
Let us use the convention that $
\mathcal{J}$ and $\mathcal{J}^{(n)}$ enter $\star$ once termination occurs, and let $L^{(n)}:=\inf\{\ell\ge 1: J(\chi^{(n)}_\ell)=\star\}$. Then 
\[Y_k^{(n)}=\sum_{\ell=0}^{L^{(n)}-1} {  {\red r\left(J^{(n)}_k(\theta_{k,\ell}^{(n)}),k\right) }}\left( \theta_{k,\ell+1}^{(n)}-\theta_{k,\ell}^{(n)}\right)=\sum_{{  \ell=0}}^{L^{(n)}-1} {  r\left(J(\chi^{(n)}_\ell), k\right)}\left( \theta_{k,\ell+1}^{(n)}-\theta_{k,\ell}^{(n)}\right).\]
Since $\{J(\chi^{(n)}_\ell)\}_{\ell\ge 0}$ is independent of $\{\theta_{k,\ell}^{(n)}\}_{\ell\ge 0}$, then (conditional on $\{J(\chi^{(n)}_\ell)\}_{\ell\ge 0}$) $Y_k^{(n)}$ is a convolution of $L^{(n)}$ exponential random variables of parameters $n/r(J(\chi^{(n)}_0),k)$, $n/r(J(\chi^{(n)}_1),k)$, $\dots$, $n/r(J(\chi^{(n)}_{L^{(n)}-1}),k)$, respectively. By counting the amount of times each state is visited, then it follows that (conditional on $\{J(\chi^{(n)}_\ell)\}_{\ell\ge 0}$ and $\mathcal{N}^{(n)}$) $Y_k^{(n)}$  follows an hypoexponential distribution resulting from convoluting $L^{(n)}_1$-$\mathrm{Exp}(n/r(1,k))$ rv's, $L^{(n)}_2$-$\mathrm{Exp}(n/r(2,k))$ rv's, $\dots$ , and $L^{(n)}_p$-$\mathrm{Exp}(n/r(p,k))$ rv's, where
\[L^{(n)}_i=\#\{\ell\ge 0: J(\chi^{(n)}_\ell)=i\},\quad i\in\mathcal{E}.\]
Thus, employing the law of total probability and the definition of the univariate densities $f_{k; i_1,\dots, i_p}$, the conditional multivariate density function associated to $(Y_1^{(n)},\dots, Y_m^{(n)})$ is of the form
\begin{align*}
f(x_1,x_2,\dots, x_m) = \sum_{\substack{(i_1,\dots, i_p)\in\mathcal{Z}}}\rho_{i_1,\dots,i_p} f_{1;i_1,\dots,i_p}(x_1)f_{2;i_1,\dots,i_p}(x_2)
  \cdots f_{m;i_1,\dots,i_p}(x_m),
\end{align*}
where 
\begin{align*}\rho_{i_1,\dots,i_p}& =\mathbb{P}\left(L^{(n)}_1=i_1,\dots, L^{(n)}_p=i_p, J(\chi^{(n)}_{|\bm{i}|})=\star\mid \mathcal{N}^{(n)}\right)\\
& = \mathbb{P}
  \left(\left\{J(\chi^{(n)}_{\ell})\right\}_{\ell= 0}^{|\bm{i}|-1}\mbox{ makes $i_a$ visits to $a$ for all $a\in\mathcal{E}$}, J(\chi^{(n)}_{|\bm{i}|})=\star\mid \mathcal{N}^{(n)}\right)\\
  & = \sum_{j\in \mathcal{E}} \mathbb{P}
  \left(\left\{J(\chi^{(n)}_{\ell})\right\}_{\ell= 0}^{|\bm{i}|-1}\mbox{ makes $i_a$ visits to $a$ for all $a\in\mathcal{E}$}, J(\chi^{(n)}_{|\bm{i}|-1})=j\mid \mathcal{N}^{(n)}\right)\\
  &\quad\quad\quad\times \mathbb{P}(J(\chi^{(n)}_{|\bm{i}|})=\star \mid  J(\chi^{(n)}_{|\bm{i}|-1})=j, \mathcal{N}^{(n)} )\\
  & = \sum_{j\in\mathcal{E}} \alpha_{i_1,\dots,i_p; j} \,\beta_{|\bm{i}|, j},\end{align*}
  from which (\ref{eq:approxKulkarni1}) follows. The r.h.s. of (\ref{eq:approxKulkarnibeta1}) and (\ref{eq:approxKulkarnialpha0}) follow by definition of $\bm{Q}^{(n)}_\ell$ and $\bm{\pi}$, respectively. Meanwhile, the recursion (\ref{eq:approxKulkarnialpharec1}) follows by noting that for $i_j>0$
  \begin{align}
 &\mathbb{P}
  \left(\left\{J(\chi^{(n)}_{\ell})\right\}_{\ell= 0}^{|\bm{i}|-1}\mbox{ makes $i_a$ visits to $a$ for all $a\in\mathcal{E}$}, J(\chi^{(n)}_{|\bm{i}|-1})=j\mid \mathcal{N}^{(n)}\right)\label{eq:approxaux7}\\
  &\quad = \sum_{h\in\mathcal{E}}\mathbb{P}
  \left(\left\{J(\chi^{(n)}_{\ell})\right\}_{\ell= 0}^{|\bm{i}|-2}\mbox{ makes $i_a-\delta_{aj}$ visits to $a$ for all $a\in\mathcal{E}$}, J(\chi^{(n)}_{|\bm{i}|-2})=h\mid \mathcal{N}^{(n)}\right)\nonumber\\
  &\quad\quad\quad\quad\times\mathbb{P}
  \left(J(\chi^{(n)}_{|\bm{i}|-1})=j \mid J(\chi^{(n)}_{|\bm{i}|-2})=h, \mathcal{N}^{(n)}\right)\nonumber\\
  &\quad = \sum_{h\in \mathcal{E}}\alpha_{i_1,\dots,i_j-1,
 \dots,i_p; h}[\bm{Q}_{|\bm{i}|-1}^{(n)}]_{hj};\nonumber
  \end{align}
given that the event $\{J(\chi^{(n)}_{|\bm{i}|-1})=j\}$ implies that there is at least one visit to $j$ on or before the $|\bm{i}|-1$ step, then the  {  expression} in the l.h.s. of (\ref{eq:approxaux7}) is null if $i_j=0$.
$\hfill\square$\end{proof}

\begin{example}\rm
Consider the Loss-ALAE dataset, comprising $n=1500$ bivariate observations. The first margin is an insurance loss, and the second one is the corresponding allocated loss adjustment
expense (ALAE). The dataset was studied in \cite{frees1998understanding}.

There are $34$ loss observations {  that are} right-censored, but to avoid deriving a new model for right-censored observations (and since the percentage of censoring is very small), we consider them as fully observed. The loss variable ranges from $10$ to $2.2$MM, with quartiles of
$4,000$, $12,000$ and $35,000$, while the ALAE variable ranges from $15$ to $0.5$MM with quartiles of $2,300$, $5,500$ and $12,600$. In a similar fashion as the analysis of \cite{joe2014dependence}, we divide all data points by a constant, in our case $10^6$, for easier numerical implementation.

The initial step consists of estimating a {  (homogeneous)} MPH$^\ast$ representation to the data. This can be done using the Expectation-Maximization algorithm implemented in the \texttt{matrixdist} package (cf. \cite{matrixdist}). The resulting estimated parameters are given by
\begin{align}\label{mph_star_params}
	\vect{\alpha}=(0.22,\, 0.73,\, 0.05),\quad \mat{S}= \left( \begin{array}{ccc}
		-11 & 0.99 & 0.09 \\
		0.11 & -46.47&0.16\\
		0.14 & 0.15&-3.21
	\end{array} \right),\quad\mat{R}=\left( \begin{array}{cc}
		0.95 & 0.05  \\
		0.56 & 0.44\\
		0.75 & 0.25
	\end{array} \right).
\end{align}

In a second step, we find the $n$ and truncation limit such that the approximation of the Markov jump process by uniformization is visually adequate. These turn out to be $n=45$ and upper truncation limit of $40$. Finally, we may visualize the MPH$^\ast$ density (which does not exist in closed-form) using the uniformization approximation, which is shown in Figure \ref{plots:kulkarni_approx}. {  Apart from likelihood considerations, the shape of the density and its visually multimodal structure can help the modeller more easily accept or challenge the fit of the model.} This insight is only possible due to the approximation, and could not be read off directly from the parameters \eqref{mph_star_params}.

\begin{figure}[!htbp]\centering
\includegraphics[width=0.49\textwidth]{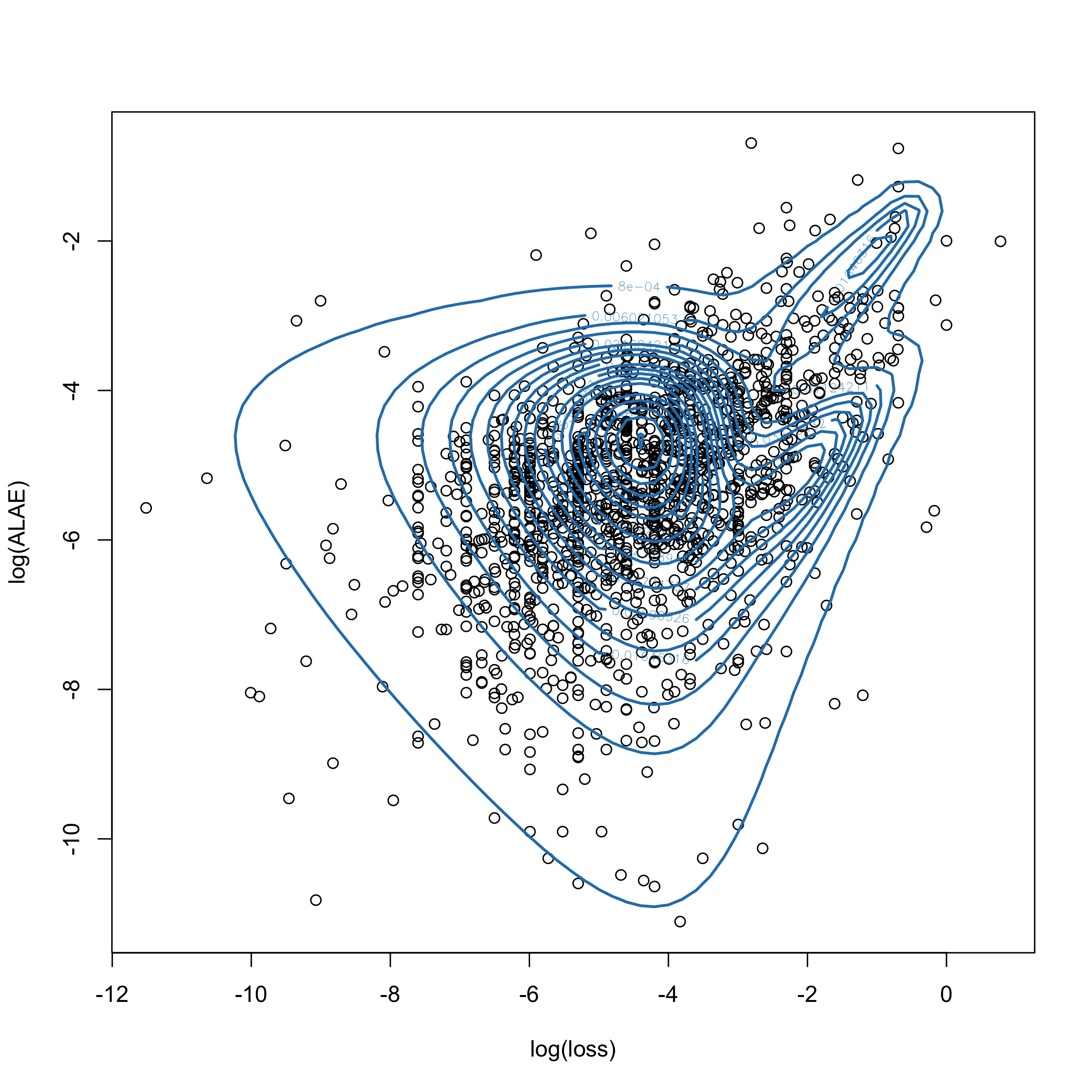}
\includegraphics[width=0.49\textwidth]{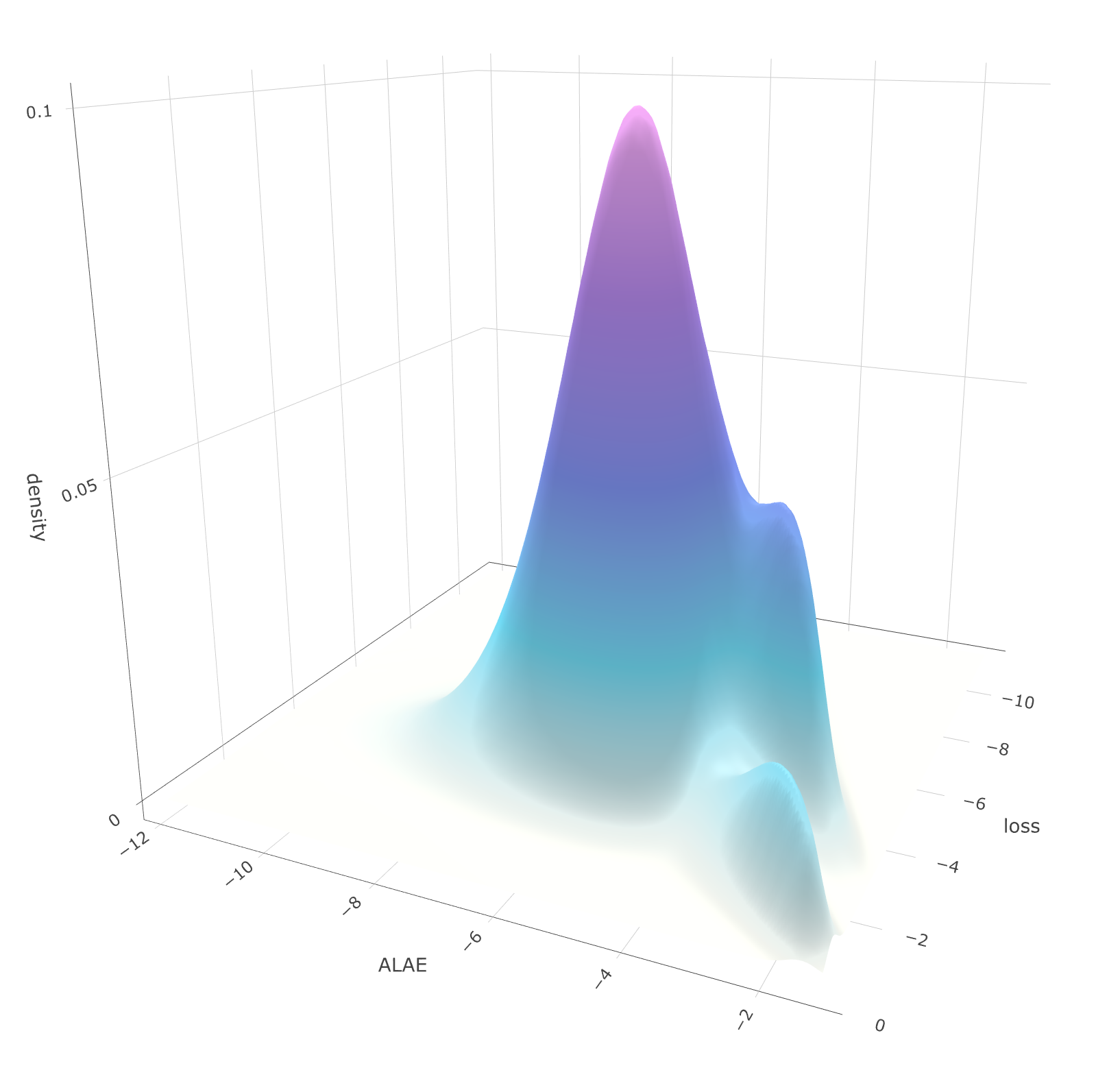}
\caption{Multivariate approximation of the $\mbox{MPH}^*$ fit to the Loss-ALAE dataset. The plot axes are log-transformed for visualization purposes.}
\label{plots:kulkarni_approx}
\end{figure}
\end{example}
  
\section{Extensions}\label{sec:extensions}

In general, most arguments presented in this article can be adapted to any jump process on a finite state-space whose jump structure are dependent on a bounded jump intensity matrix function. For instance, one can build strong approximations to the completely non-homogeneous semi-Markov process considered in \cite[Chapter 6]{janssen2006applied}, a model which encompasses characteristics of both time-inhomogeneous and semi-Markovian models. However, due to the complexity of the said model, the approximations will not be nearly as tractable as the ones presented here for the {\red time-inhomogeneous case}.

\section*{Acknowledgments.}
MB and OP would like to acknowledge financial support from the Swiss National Science Foundation Project 200021\_191984.



\bibliographystyle{informs2014} 
\bibliography{strongPH.bib} 

\begin{thebibliography}{56}
\providecommand{\natexlab}[1]{#1}
\providecommand{\url}[1]{\texttt{#1}}
\providecommand{\urlprefix}{URL }

\bibitem[{Albrecher \protect\BIBand{} Bladt(2019)}]{albrecher2019inhomogeneous}
Albrecher H, Bladt M (2019) Inhomogeneous phase-type distributions and heavy
  tails. \emph{Journal of Applied Probability} 56(4):1044--1064.

\bibitem[{Albrecher et~al.(2022)Albrecher, Bladt, \protect\BIBand{}
  Yslas}]{albrecher2022fitting}
Albrecher H, Bladt M, Yslas J (2022) Fitting inhomogeneous phase-type
  distributions to data: the univariate and the multivariate case.
  \emph{Scandinavian Journal of Statistics} 49(1):44--77.

\bibitem[{Arns et~al.(2010)Arns, Buchholz, \protect\BIBand{}
  Panchenko}]{arns2010numerical}
Arns M, Buchholz P, Panchenko A (2010) On the numerical analysis of
  inhomogeneous continuous-time {M}arkov chains. \emph{INFORMS Journal on
  Computing} 22(3):416--432.

\bibitem[{Asmussen(2003)}]{asmussen2003applied}
Asmussen S (2003) \emph{Applied probability and queues}, volume~2 (Springer).

\bibitem[{Asmussen \protect\BIBand{} Albrecher(2010)}]{asmussen2010ruin}
Asmussen S, Albrecher H (2010) \emph{Ruin probabilities}, volume~14 (World
  scientific).

\bibitem[{Asmussen et~al.(1996)Asmussen, Nerman, \protect\BIBand{}
  Olsson}]{asmussen1996fitting}
Asmussen S, Nerman O, Olsson M (1996) Fitting phase-type distributions via the
  {EM} algorithm. \emph{Scandinavian Journal of Statistics} 419--441.

\bibitem[{Assaf et~al.(1984)Assaf, Langberg, Savits, \protect\BIBand{}
  Shaked}]{assaf1984multivariate}
Assaf D, Langberg NA, Savits TH, Shaked M (1984) Multivariate phase-type
  distributions. \emph{Operations Research} 32(3):688--702.

\bibitem[{Bladt(2023)}]{bladt2021tractable}
Bladt M (2023) A tractable class of multivariate phase-type distributions for
  loss modeling. \emph{North American Actuarial Journal, to appear} .

\bibitem[{Bladt et~al.(2020)Bladt, Asmussen, \protect\BIBand{}
  Steffensen}]{bladt2020matrix}
Bladt M, Asmussen S, Steffensen M (2020) Matrix representations of life
  insurance payments. \emph{European Actuarial Journal} 10(1):29--67.

\bibitem[{Bladt et~al.(2003)Bladt, Gonzalez, \protect\BIBand{}
  Lauritzen}]{bladt2003estimation}
Bladt M, Gonzalez A, Lauritzen SL (2003) The estimation of phase-type related
  functionals using {M}arkov chain {M}onte {C}arlo methods. \emph{Scandinavian
  Actuarial Journal} 2003(4):280--300.

\bibitem[{Bladt \protect\BIBand{} Nielsen(2010)}]{bladt2010multivariate}
Bladt M, Nielsen BF (2010) Multivariate matrix-exponential distributions.
  \emph{Stochastic Models} 26(1):1--26.

\bibitem[{Bladt \protect\BIBand{} Nielsen(2017)}]{bladt2017matrix}
Bladt M, Nielsen BF (2017) \emph{Matrix-exponential distributions in applied
  probability}, volume~81 (Springer).

\bibitem[{Bladt et~al.(2015)Bladt, Nielsen, \protect\BIBand{}
  Samorodnitsky}]{bladt2015calculation}
Bladt M, Nielsen BF, Samorodnitsky G (2015) Calculation of ruin probabilities
  for a dense class of heavy tailed distributions. \emph{Scandinavian Actuarial
  Journal} 2015(7):573--591.

\bibitem[{Bladt \protect\BIBand{} Yslas(2022)}]{matrixdist}
Bladt M, Yslas J (2022) matrixdist: Statistics for matrix distributions.
  \url{github.com/martinbladt/matrixdist_1.0}.

\bibitem[{Breuer(2016)}]{breuer2016semi}
Breuer L (2016) A semi-explicit density function for {K}ulkarni's bivariate
  phase-type distribution. \emph{Stochastic Models} 32(4):632--642.

\bibitem[{Choudhury et~al.(1997)Choudhury, Mandelbaum, Reiman,
  \protect\BIBand{} Whitt}]{choudhury1997fluid}
Choudhury G, Mandelbaum A, Reiman M, Whitt W (1997) Fluid and diffusion limits
  for queues in slowly changing environments. \emph{Stochastic Models}
  13(1):121--146.

\bibitem[{Cloth et~al.(2007)Cloth, Jongerden, \protect\BIBand{}
  Haverkort}]{cloth2007computing}
Cloth L, Jongerden MR, Haverkort BR (2007) Computing battery lifetime
  distributions. \emph{37th Annual IEEE/IFIP International Conference on
  Dependable Systems and Networks (DSN'07)}, 780--789 (IEEE).

\bibitem[{Cramer(1955)}]{cramer1955collective}
Cramer H (1955) \emph{Collective risk theory} (Skandia ins. Company).

\bibitem[{Cs{\"o}rgo \protect\BIBand{} R{\'e}v{\'e}sz(2014)}]{csorgo2014strong}
Cs{\"o}rgo M, R{\'e}v{\'e}sz P (2014) \emph{Strong approximations in
  probability and statistics} (Academic Press).

\bibitem[{Dollard \protect\BIBand{} Friedman(1984)}]{Dollard1984ProductIW}
Dollard JD, Friedman CN (1984) Product integration with applications to
  differential equations: Product integration of measures.

\bibitem[{Fraser(1973)}]{fraser1973rate}
Fraser DF (1973) {The rate of convergence of a random walk to Brownian motion}.
  \emph{The Annals of Probability} 699--701.

\bibitem[{Frees \protect\BIBand{} Valdez(1998)}]{frees1998understanding}
Frees EW, Valdez EA (1998) Understanding relationships using copulas.
  \emph{North American Actuarial Journal} 2(1):1--25.

\bibitem[{Frenkel(2009)}]{frenkel2009non}
Frenkel ALI (2009) Non-homogeneous {M}arkov reward model for aging multi-state
  system under minimal repair. \emph{International Journal of Performability
  Engineering} 5(4):303.

\bibitem[{Friz \protect\BIBand{} Hairer(2020)}]{friz2020course}
Friz PK, Hairer M (2020) \emph{A course on rough paths} (Springer).

\bibitem[{Gill(1994)}]{gill1994lectures}
Gill RD (1994) Lectures on survival analysis. \emph{Lectures on Probability
  Theory}, 115--241 (Springer).

\bibitem[{Gill \protect\BIBand{} Johansen(1987)}]{gill1987product}
Gill RD, Johansen S (1987) Product-integrals and counting processes.
  \emph{Department of Mathematical Statistics} (R 8707).

\bibitem[{Grassmann(1977)}]{grassmann1977transient}
Grassmann WK (1977) Transient solutions in {M}arkovian queueing systems.
  \emph{Computers \& Operations Research} 4(1):47--53.

\bibitem[{Haase(2006)}]{haase2006functional}
Haase M (2006) The functional calculus for sectorial operators. \emph{The
  Functional Calculus for Sectorial Operators}, 19--60 (Springer).

\bibitem[{Hampshire et~al.(2006)Hampshire, Harchol-Balter, \protect\BIBand{}
  Massey}]{hampshire2006fluid}
Hampshire RC, Harchol-Balter M, Massey WA (2006) Fluid and diffusion limits for
  transient sojourn times of processor sharing queues with time varying rates.
  \emph{Queueing Systems} 53:19--30.

\bibitem[{Helton \protect\BIBand{} Stuckwisch(1976)}]{helton1976numerical}
Helton J, Stuckwisch S (1976) Numerical approximation of product integrals.
  \emph{Journal of Mathematical Analysis and Applications} 56(2):410--437.

\bibitem[{Jacod \protect\BIBand{} Shiryaev(2013)}]{jacod2013limit}
Jacod J, Shiryaev A (2013) \emph{Limit theorems for stochastic processes},
  volume 288 (Springer Science \& Business Media).

\bibitem[{Janssen \protect\BIBand{} Manca(2006)}]{janssen2006applied}
Janssen J, Manca R (2006) \emph{Applied semi-{M}arkov processes} (Springer
  Science \& Business Media).

\bibitem[{Jensen(1953)}]{jensen1953markoff}
Jensen A (1953) Markoff chains as an aid in the study of {M}arkoff processes.
  \emph{Scandinavian Actuarial Journal} 1953:87--91.

\bibitem[{Joe(2014)}]{joe2014dependence}
Joe H (2014) \emph{Dependence modeling with copulas} (CRC press).

\bibitem[{Kulkarni(1989)}]{kulkarni1989new}
Kulkarni VG (1989) A new class of multivariate phase-type distributions.
  \emph{Operations Research} 37(1):151--158.

\bibitem[{Kurtz(1971)}]{kurtz1971limit}
Kurtz TG (1971) {Limit theorems for sequences of jump Markov processes
  approximating ordinary differential processes}. \emph{Journal of Applied
  Probability} 8(2):344--356.

\bibitem[{Lawless(2011)}]{lawless2011statistical}
Lawless JF (2011) \emph{Statistical models and methods for lifetime data} (John
  Wiley \& Sons).

\bibitem[{Lundberg(1903)}]{lundberg1903approximerad}
Lundberg F (1903) \emph{Approximerad framst{\"a}llning av
  sannolikhetsfunktionen. aterf{\"o}rs{\"a}kring av kollektivrisker. Akad}.
  Ph.D. thesis, Almqvist och Wiksell Uppsala, Sweden.

\bibitem[{Mandelbaum \protect\BIBand{} Massey(1995)}]{mandelbaum1995strong}
Mandelbaum A, Massey WA (1995) Strong approximations for time-dependent queues.
  \emph{Mathematics of Operations Research} 20(1):33--64.

\bibitem[{Mandelbaum et~al.(1998)Mandelbaum, Massey, \protect\BIBand{}
  Reiman}]{mandelbaum1998strong}
Mandelbaum A, Massey WA, Reiman MI (1998) {Strong approximations for Markovian
  service networks}. \emph{Queueing Systems} 30:149--201.

\bibitem[{Massey(1985)}]{massey1985asymptotic}
Massey WA (1985) {Asymptotic analysis of the time dependent M/M/1 queue}.
  \emph{Mathematics of Operations Research} 10(2):305--327.

\bibitem[{Massey \protect\BIBand{} Whitt(1998)}]{massey1998uniform}
Massey WA, Whitt W (1998) {Uniform acceleration expansions for Markov chains
  with time-varying rates}. \emph{Annals of Applied Probability} 1130--1155.

\bibitem[{Moler \protect\BIBand{} Van~Loan(2003)}]{moler2003nineteen}
Moler C, Van~Loan C (2003) Nineteen dubious ways to compute the exponential of
  a matrix, twenty-five years later. \emph{SIAM Review} 45(1):3--49.

\bibitem[{Neuts(1975)}]{neuts1975probability}
Neuts MF (1975) Probability distributions of phase-type. \emph{Liber Amicorum
  Prof. Emeritus H. Florin} .

\bibitem[{Nguyen \protect\BIBand{} Peralta(2022)}]{nguyen2022rate}
Nguyen GT, Peralta O (2022) Rate of strong convergence to {M}arkov-modulated
  {B}rownian motion. \emph{Journal of Applied Probability} 1–16.

\bibitem[{Prokhorov(1956)}]{prokhorov1956convergence}
Prokhorov YV (1956) Convergence of random processes and limit theorems in
  probability theory. \emph{Theory of Probability \& Its Applications}
  1(2):157--214.

\bibitem[{Rindos et~al.(1995)Rindos, Woolet, Viniotis, \protect\BIBand{}
  Trivedi}]{rindos1995exact}
Rindos A, Woolet S, Viniotis I, Trivedi K (1995) Exact methods for the
  transient analysis of nonhomogeneous continuous time {M}arkov chains.
  \emph{Computations with Markov chains}, 121--133 (Springer).

\bibitem[{Shi et~al.(2005)Shi, Guo, \protect\BIBand{} Liu}]{shi2005sph}
Shi D, Guo J, Liu L (2005) On the {SPH}-distribution class. \emph{Acta
  Mathematica Scientia} 25(2):201--214.

\bibitem[{Skorokhod(1956)}]{skorokhod1956limit}
Skorokhod AV (1956) Limit theorems for stochastic processes. \emph{Theory of
  Probability \& Its Applications} 1(3):261--290.

\bibitem[{Strassen(1964)}]{strassen1964invariance}
Strassen V (1964) An invariance principle for the law of the iterated
  logarithm. \emph{Zeitschrift f{\"u}r Wahrscheinlichkeitstheorie und verwandte
  Gebiete} 3(3):211--226.

\bibitem[{Telek et~al.(2004)Telek, Horvath, \protect\BIBand{}
  Horv{\'a}th}]{telek2004analysis}
Telek M, Horvath A, Horv{\'a}th G (2004) Analysis of inhomogeneous {M}arkov
  reward models. \emph{Linear Algebra and its Applications} 386:383--405.

\bibitem[{van Dijk(1992)}]{van1992uniformization}
van Dijk NM (1992) Uniformization for nonhomogeneous {M}arkov chains.
  \emph{Operations Research Letters} 12(5):283--291.

\bibitem[{van Dijk et~al.(2018)van Dijk, van Brummelen, \protect\BIBand{}
  Boucherie}]{van2018uniformization}
van Dijk NM, van Brummelen SPJ, Boucherie RJ (2018) Uniformization: {B}asics,
  extensions and applications. \emph{Performance evaluation} 118:8--32.

\bibitem[{van Moorsel \protect\BIBand{} Wolter(1998)}]{van1998numerical}
van Moorsel AP, Wolter K (1998) Numerical solution of non-homogeneous {M}arkov
  processes through uniformization. \emph{ESM}, 710--717 (Citeseer).

\bibitem[{Whitt(2002)}]{whitt2002stochastic}
Whitt W (2002) Stochastic-process limits: an introduction to stochastic-process
  limits and their application to queues. \emph{Space} 500:391--426.

\bibitem[{Wong \protect\BIBand{} Zakai(1965)}]{wong1965relation}
Wong E, Zakai M (1965) On the relation between ordinary and stochastic
  differential equations. \emph{International Journal of Engineering Science}
  3(2):213--229.

\end{thebibliography}


\end{document}